\documentclass[12 pt]{amsart}

\usepackage[letterpaper,top=2cm,bottom=2cm,left=3cm,right=3cm,marginparwidth=1.75cm]{geometry}
\usepackage{graphicx} 
\usepackage{amssymb,amsthm,amsmath,eucal,latexsym,tikz}
\usepackage{color,caption,subcaption,booktabs,adjustbox,comment}
\usepackage{xcolor,verbatim}

\newtheorem{theorem}{Theorem}[section]
\newtheorem{corollary}[theorem]{Corollary}

\newtheorem{lemma}[theorem]{Lemma}

\DeclareMathOperator{\CL}{CL}
\DeclareMathOperator{\diam}{diam}

\newcommand{\vA}{{u}} 
\newcommand{\vB}{{v}} 
\newcommand{\vC}{{s}} 
\newcommand{\mA}{{b}} 
\newcommand{\mB}{{c}} 
\newcommand{\mC}{{e}} 
\newcommand{\cv}{{V}} 
\newcommand{\numsteps}{{W}}
\newcommand{\mBox}{
\mathop{\vcenter{\hbox{\scalebox{1.2}{$\square$}}}}\nolimits}

\definecolor{trentviolet}{RGB}{148,0,211}

\title{Cooling graph products}
\author[A.\ Bonato]{Anthony Bonato}
\author[M.\ Carr]{MacKenzie Carr}
\author[C.\ Jones]{Caleb Jones}
\author[T.G.\ Marbach]{Trent G.\ Marbach}
\author[T.\ Mishura]{Teddy Mishura}

\dedicatory{In memory of Robert Woodrow}

\address[A1,A2,A3,A5]{Toronto Metropolitan University, Toronto, Canada}
\address[A4]{Acadia University, Wolfville, Canada}

\email[A1]{(A1) abonato@torontomu.ca}
\email[A2]{(A2) mackenzie.carr@torontomu.ca}
\email[A3]{(A3) caleb.w.jones@torontomu.ca}
\email[A4]{(A4) trent.marbach@acadiau.ca}
\email[A5]{(A5) tmishura@torontomu.ca}


\begin{document}

\begin{abstract} 

The cooling number measures the speed at which a slow-moving influence or contagion spreads on a graph. In this paper, we investigate the cooling number of four classical graph products: the Cartesian product, the strong product, the lexicographic product, and the direct product. We also determine the cooling number of a disconnected graph in terms of the cooling numbers of its components. We conclude with open problems. 

\end{abstract}

\maketitle

\section{Introduction}

Graph burning is a simplified model for the spread of influence or contagion in a network. The burning number of a graph, introduced in \cite{BJR0,BJR}, quantifies the speed at which the influence spreads to every vertex in the fastest case. Given a graph $G$, the burning process on $G$ is a discrete-time process that proceeds in rounds. Throughout the process, vertices are \emph{burned} or \emph{unburned}. At the beginning of the first round, all vertices are unburned. In each round, all unburned vertices that have a burned neighbor become burned (the \emph{propagation} step), and then one new unburned vertex (if available) is chosen to burn (the \emph{manual} step); the latter vertices are called \emph{sources}. The sequence of sources chosen during the process is called a \emph{burning sequence}. The \emph{burning number} of $G,$ written $b(G),$ is the minimum number of rounds needed until all vertices are burned. Burning has since been extensively studied in papers such as \cite{Burning_Hard,bessy2,bonf,bgs,bk,bl,bill,rj,kam,ko,LL,liu1,MitschePralatRoshanbin,prod}. For further background on graph burning, see the survey \cite{survey} and the book \cite{bbook}.

A recent variant of burning called \emph{cooling} was introduced in \cite{cooling} as a measure of the speed of the slowest case spread of a contagion or influence in a network. Cooling is defined analogously to burning, except that the \emph{cooling number} of a graph $G$, written $\CL(G),$ is the maximum number of rounds needed until all vertices are cooled. Throughout the process, vertices are \emph{cooled} or \emph{uncooled}. Note that \emph{cooling sequences} are defined analogously to burning sequences; in particular, they are defined as the sequence of vertices that are chosen as sources. We say that a cooling sequence for a graph $G$ is \emph{optimal} if it realizes the cooling number of $G$. Observe that finding a sequence of sources in $G$ that lasts $k$ rounds yields $b(G)\leq k\leq \CL(G)$. In particular, finding a specific burning sequence yields an \emph{upper bound} on $b(G)$, whereas finding a specific cooling sequence yields a \emph{lower bound} on $\CL(G)$. Cooling has subsequently been studied in the burning game in \cite{bill}, which mixes cooling and burning, and in liminal burning, which provides burning-type parameters that interpolate between the cooling and burning numbers \cite{liminal}.

In this paper, we study cooling on the following graph products: the Cartesian, strong, lexicographic, and direct products (also called the tensor product or categorical product). While such products have been extensively studied (see the book \cite{prod11} for further background), we define them here for completeness. Given graphs $G$ and $H$ (the \emph{factors} of the product), the vertex set of each product is $V(G) \times V(H).$ The edge set for each graph product, along with its notation, is presented in the following table. See Figure~\ref{fig:INTRO:fig1} for visualizations of the four products of $P_3$ with itself.\\

\begin{center}
\renewcommand{\arraystretch}{1.5}
\begin{tabular}{|l|c|p{10cm}|}
\hline
Cartesian & \(G \square H\) &
\(E(G \square H)=\{(u,x)(v,y)\in V\times V : u=v,\ xy\in E(H)\}\cup\{(u,x)(v,y)\in V\times V : x=y,\ uv\in E(G)\}.\) \\
\hline
direct & \(G \times H\) &
\(E(G \times H)=\{(u,x)(v,y)\in V\times V : uv\in E(G),\ xy\in E(H)\}.\) \\
\hline
strong & \(G \boxtimes H\) &
\(E(G \boxtimes H)=E(G \square H)\cup E(G \times H).\) \\
\hline
lexicographic & \(G \circ H\) &
\(E(G \circ H)=\{(u,x)(v,y)\in V\times V : uv\in E(G)\}\cup\{(u,x)(v,y)\in V\times V : u=v,\ xy\in E(H)\}.\) \\
\hline
\end{tabular}
\end{center}

\begin{figure}[ht]
\centering

\begin{tikzpicture}[
    scale=0.75,
    every node/.style={circle, fill=black, inner sep=2.2pt},
    every path/.style={draw=black, line width=0.6pt}
]

\newcommand{\gridnodes}[1]{
    \foreach \i in {0,1,2}
        \foreach \j in {0,1,2}
            \node (#1-\i-\j) at (\i,\j) {};
}

\newcommand{\cartesianedges}[1]{
    \foreach \j in {0,1,2} {
        \draw (#1-0-\j)--(#1-1-\j);
        \draw (#1-1-\j)--(#1-2-\j);
    }
    \foreach \i in {0,1,2} {
        \draw (#1-\i-0)--(#1-\i-1);
        \draw (#1-\i-1)--(#1-\i-2);
    }
}
\begin{scope}[shift={(0,0)}]
\gridnodes{C}
\cartesianedges{C}
\node[fill=none] at (1,-0.8) {$P_3 \square P_3$};
\end{scope}

\begin{scope}[shift={(4,0)}]
\gridnodes{S}
\cartesianedges{S}
\foreach \i in {0,1}
    \foreach \j in {0,1} {
        \draw (S-\i-\j)--(S-\the\numexpr\i+1\relax-\the\numexpr\j+1\relax);
        \draw (S-\i-\the\numexpr\j+1\relax)--(S-\the\numexpr\i+1\relax-\j);
    }
\node[fill=none] at (1,-0.8) {$P_3 \boxtimes P_3$};
\end{scope}

\begin{scope}[shift={(8,0)}]
\gridnodes{L}
\foreach \i in {0,1,2} {
    \draw (L-\i-0)--(L-\i-1);
    \draw (L-\i-1)--(L-\i-2);
}
\foreach \i/\k in {0/1,1/2}
    \foreach \j in {0,1,2}
        \foreach \ell in {0,1,2}
            \draw (L-\i-\j)--(L-\k-\ell);

\node[fill=none] at (1,-0.8) {$P_3 \circ P_3$};
\end{scope}

\begin{scope}[shift={(12,0)}]
\gridnodes{D}
\foreach \i in {0,1}
    \foreach \j in {0,1} {
        \draw (D-\i-\j)--(D-\the\numexpr\i+1\relax-\the\numexpr\j+1\relax);
        \draw (D-\i-\the\numexpr\j+1\relax)--(D-\the\numexpr\i+1\relax-\j);
    }
\node[fill=none] at (1,-0.8) {$P_3 \times P_3$};
\end{scope}
\end{tikzpicture}
\caption{The four main graph products with both factors equaling $P_3$.}\label{fig:INTRO:fig1}
\end{figure}
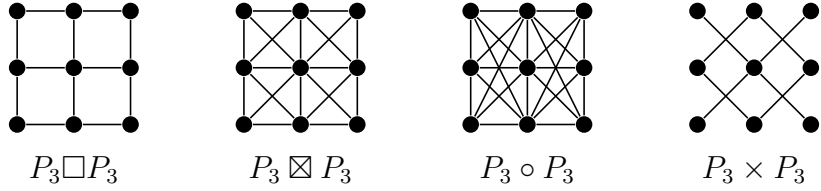

Burning Cartesian, strong, and lexicographic products of graphs was first considered in \cite{MitschePralatRoshanbin,prod}. Interestingly, determining the burning number of $m \times n$ Cartesian grids remains an open problem, with partial results given in \cite{bonf,MitschePralatRoshanbin} for so-called fence graphs. Cooling square strong grids was fully solved in \cite{ambrose}, and we generalize this result to rectangular strong grids in Theorem~\ref{thm:STRONG:strong_grid_result}. To the best of our knowledge, no study has considered the burning number of direct products.

Our present focus is on establishing bounds and exact values of the cooling numbers of the four main graph products. While cooling is somewhat analogous to burning, in the latter process we often exploit decompositions such as rooted tree partitions to bound the burning number; by contrast, no such decompositions exist in general for cooling, so other techniques must be employed. 

The paper is organized as follows. 
In Section~2, we begin by collecting elementary but useful bounds on the cooling number of graphs, and introduce the notion of maximally cool graphs.
In Section~3, we study the Cartesian product of graphs, focusing on products of trees and complete graphs. We give general lower bounds on the cooling number of $G\square H$ and identify graphs whose product has cooling number as large as possible with respect to their diameter. 
In Section~4, we obtain upper and lower bounds for the cooling number of strong product of general graphs. Using these bounds, we completely settle the cooling number of square strong grids and obtain strong bounds on the cooling number of rectangular strong grids.
In Section~5, we show that cooling behaves differently from burning on the lexicographic products of graphs compared to the other types of products. We obtain an upper bound for the cooling number of the product in terms of the cooling number of its factors that differs drastically from that of the burning number.
In Section~6, we show that it is impossible to create an upper bound for the cooling number of the direct product of general graphs in terms of the cooling numbers of its factors. We also show that cooling a disconnected graph reduces to optimally cooling each of its components in sequence.
We finish with open problems. Throughout the paper, we include computational results obtained using code available at \texttt{https://github.com/calebwjones/graph\_cooling\_code.git}. 

We consider only simple, finite, and undirected graphs. All graphs are assumed to be connected unless otherwise indicated. For a positive integer $n$, we let $[n] = \{ 1, 2, \ldots ,n\}$. Unless otherwise stated, variables $k$, $m,$ and $n$ refer to integers. In a graph $G$, we denote the distance between vertices $u$ and $v$ by $d_G(u,v)$, writing $d(u,v)$ if $G$ is clear from context. The \emph{diameter} of $G$ is $\diam(G)=\max_{u,v \in V(G)}d(u,v).$ The \emph{complete graph}, the \emph{path graph} and the \emph{cycle graph}, each of order $n$, are denoted by $K_n$, $P_n$ and $C_n$, respectively. The graph with no edges of order $n$ is denoted by $\overline{K}_n$. \emph{Caterpillars} are trees consisting of a path $P$, called the \emph{spine}, with any number of leaves (that is, vertices of degree $1$), adjacent to $P$. We refer to the non-spine leaves of a caterpillar as its \emph{legs}. \emph{Centipedes} are caterpillars with maximum degree 3. A centipede is \emph{full} if every non-leaf vertex has degree exactly 3, also called a complete caterpillar in \cite{cooling}. For background on graph theory, see \cite{west}. For additional background on burning and other pursuit-evasion processes and games, see \cite{bbook}.

\section{Preliminary results on cooling}

We begin with the basic bounds on the cooling number of $G$ in terms of its diameter that were established in \cite{cooling}.  

\begin{theorem}[{\cite[Theorem 2]{cooling}}]
\label{thm:GENE:diam_bounds}
If $G$ is a graph, then $$\left\lceil\frac{\diam(G)+2}{2}\right\rceil\leq \CL(G)\leq \diam(G)+1.$$
\end{theorem}

Note that these bounds lead to an immediate corollary bounding the diameter in terms of the cooling number. This will later be used to prove Theorem \ref{thm:LEX:circ_general}.

\begin{corollary}
\label{cor:GENE:another_diam_bound}
If $G$ is a graph, then $\mathrm{CL}(G) -1 \leq \diam(G)\leq 2\CL(G)-2$.
\end{corollary}

Because cooling a graph $G$ requires that all sources be uncooled at the time of their manual cooling, all sources in a cooling sequence must be far enough apart from each other in $G$ to avoid cooling the others through propagation. It is therefore often useful to think about cooling as a graph covering or packing problem, where we wish to cover $G$ in as many balls of decreasing radius (whose centers do not touch) as possible. 

To aid in this mode of thought, we recall the following terminology. For $i\geq 0$, denote by $N_i[v]$ the set of all vertices with distance at most $i$ to $v$. More formally, $N_i[v] = \{u\in V(G) : d(v,u) \leq i\}$, where $v\in N_i[v]$ for all $i$. For $S \subseteq V(G)$, let $N_i[S] = \{u \in V(G):d(u,v) \leq i\text{ for some } v \in S\}.$ When $i=1$, we write $N[v]$ and $N[S]$ instead. We can now rephrase cooling entirely in terms of these objects.

\begin{lemma}
\label{lem:GENE:cooling_sequence_equation}
A sequence $(\vA_j)$ of vertices in a graph $G$ is a valid sequence of sources if and only if $d_G(\vA_i,\vA_j)\geq |j-i|+1$.  
\end{lemma}
\begin{proof}
    Let $S_j$ denote the set of cooled vertices at the end of round $j$ using the sequence $(\vA_j)$, 
    noting that $S_1=\{\vA_1\}$ and that $S_j=\bigcup_{i=1}^j N_{j-i}[\vA_i]$. In round $j$, the set of cooled vertices after the propagation step is $N[S_{j-1}]=\bigcup_{i=1}^{j-1} N[N_{j-1-i}[\vA_{i}]]$ $=\bigcup_{i=1}^{j-1} N_{j-i}[\vA_i]$. We manually cool $\vA_{j}$ so that $N[S_{j-1}]\cup\{\vA_{j}\} = \bigcup_{i=1}^{j} N_{j-i}[\vA_i]=S_{j}$, noting that it must be uncooled in round $j$. 
    Thus, $\vA_{j}\notin \bigcup_{i=1}^{j-1} N_{j-i}[\vA_i]$, implying that $d(\vA_i,\vA_j)>j-i$ for $1\leq i< j$, and so $d(\vA_i,\vA_j)\geq j-i+1$ for all $\vA_i,\vA_j$ in the sequence. 

    If $d(\vA_i,\vA_j)\geq j-i+1$ for each $i<j$, then $\vA_{j}\notin \bigcup_{i=1}^{j-1} N_{j-i}[\vA_i]$. 
    This holds for each $j$, so $(\vA_j)$ is a valid sequence of sources. 
\end{proof}

A cooling sequence must therefore be a maximal vertex sequence $(\vA_j)$ satisfying $d_G(\vA_i,\vA_j)\geq |j-i|+1$, which we call the \emph{cooling condition}. This rephrasing allows for analysis of cooling purely in terms of the distances between sequence elements, thereby introducing useful new tools. One such tool is an \emph{isometric} subgraph $G'$ of a graph $G$, which is an induced subgraph of $G$ such that $d_{G'}(u,v)=d_G(u,v)$ for all $u,v\in V(G')$. 

Note that a cooling sequence of maximum length may not necessarily be an optimal cooling sequence. In particular, if a source is chosen in the final round, then the number of sources matches the number of rounds. However, it is often the case that no source is available after propagation in the final round; hence, the number of rounds is one more than the length of the cooling sequence. It is therefore possible to have two sequences of the same length where one lasts an extra round.   
 
\begin{lemma} \label{lem:GENE:isometric_subgraph}
If $G'$ is an isometric subgraph of $G$, then
$\CL(G)\geq\CL(G')
$.
\end{lemma}

\begin{proof}
Consider an optimal cooling sequence $(\vA_1,\vA_2,\ldots, \vA_\mA)$ of $G'$. By Lemma \ref{lem:GENE:cooling_sequence_equation}, we have that $d_G(\vA_i,\vA_j) = d_{G'}(\vA_i,\vA_j) \geq |j-i|+1$ and thus, $(\vA_1,\vA_2,\ldots, \vA_\mA)$ is also a cooling sequence for $G$. If $\CL(G')=b,$ then we are done. If $\CL(G')=b+1$, then some $v\in V(G')$ is uncooled at the end of round $b$, and $d_{G'}(v,u_i) = d_G(v,u_i) \geq b-i+1$ for $1\leq i \leq b$. Thus, $v$ is uncooled at the end of round $b$ in $G$ and $\CL(G) \geq \CL(G').$ 
\end{proof}

We define the useful concept of a \emph{maximally cool} graph, which is a graph $G$ with $\CL(G) = \diam(G)+1$. 
As we prove in the following result, such graphs must contain at least one vertex that is a diameter distance from \emph{multiple} other vertices.

\begin{theorem}\label{thm:GENE:eccentricity}
Let $G$ be a graph of diameter $d$. If for all $v \in V(G)$ we have that $|N_d[v]\setminus N_{d-1}[v]| \leq 1,$ then $G$ is \emph{not} maximally cool.
\end{theorem}

\begin{proof}
Suppose that $G$ is a graph satisfying $|N_d[v]\setminus N_{d-1}[v]| \leq 1$ for every $v\in V(G)$, and consider the cooling process in $G$, where $S_j$ is the set of cooled vertices at the end of round $j$. Let $S_1 = \{v\}$. We show that cooling in $G$ can last at most $d$ rounds.

If $v$ does not have maximum eccentricity, then $V(G)=N_{d-1}[v] \subseteq S_d$ and $G$ is cooled by the end of round $d$. Thus, $\CL(G) \leq d$, so $G$ is not maximally cool. 

If $v$ has maximum eccentricity, then $N_d[v] \setminus N_{d-1}[v] =\{u\}$ for some $u$. Therefore, at the end of the propagation step in round $d$, all of $N[S_{d-1}]$ cools, leaving only $u$ uncooled. We then manually cool $u$ in round $d$, so $\CL(G) \leq d$ and $G$ is not maximally cool.
\end{proof}

\section{Cooling Cartesian products}

We now turn to the cooling number of Cartesian products of graphs. 
Theorem~\ref{thm:GENE:diam_bounds} shows that the cooling number of any graph is bounded below by roughly half of its diameter. 
Our first result substantially improves this lower bound for the Cartesian product of general graphs, showing that the product's cooling number is at most an additive logarithmic term away from the product's diameter. 
We will often denote the factors as $G$ and $H$, and for notational convenience, we let $D_G=\diam(G)$, let $D_H=\diam(H)$, and let $D_\square=\diam(G\square H)=D_G+D_H$.

\begin{theorem}\label{thm:CART:primary_cartesian_bounds}
For graphs $G$ and $H$ with $D_G\leq D_H$, we have 
\[
D_\square+1 - 2\left\lfloor\log_2(D_G+2)\right\rfloor - \left\lfloor \frac{D_H-D_G+2}{D_G+2}\right\rfloor\leq \CL(G\square H) \leq D_\square+1.
\]
\end{theorem}

The upper bound of Theorem~\ref{thm:CART:primary_cartesian_bounds} follows directly from Theorem~\ref{thm:GENE:diam_bounds}. 
To prove the lower bound, we first require two lemmas.  

We adapt Lemma~\ref{lem:GENE:isometric_subgraph} to Cartesian products of graphs, showing that the product's cooling number is related to the cooling number of the product of isometric subgraphs of each of the factors.

\begin{lemma}\label{lem:CART:isometric_subgraph_cartesian}
If $G'$ and $H'$ are isometric subgraphs of $G$ and $H$, respectively, then
$$
\CL(G\square H ) \geq \CL(G' \square H')
.$$
\end{lemma}

\begin{proof}
Let $(u,v),(x,y) \in V(G'\square H')$. We then have that 
$$d_{G'\square H'}((u,v),(x,y))
= d_{G'}(u,x) + d_{H'}(v,y)
= d_{G}(u,x)+d_{H}(v,y)
=d_{G\square H}((u,v),(x,y)).$$
Thus, $G'\square H'$ is an isometric subgraph of $G\square H$. It follows from Lemma \ref{lem:GENE:isometric_subgraph} that $\CL(G \square H) \geq \CL(G' \square H')$. 
\end{proof}

Lemma~\ref{lem:CART:isometric_subgraph_cartesian} provides good bounds on the cooling number of complex graphs by considering the cooling number of their elementary subgraphs. In particular, because every graph $G$ has an isometric path with $\diam(G)+1$ vertices, any lower bound concerning $\CL(P_m\square P_n)$ is a lower bound on $\CL(G\square H)$ for appropriate graphs $G$ and $H$. We present such a bound below, noting that it is a generalization of a lower bound for $\CL(P_n\square P_n)$ found in \cite{cooling}.   
\begin{lemma} \label{lem:CART:cartesian_rectangular_grids}
If $m\le n$, then
$\CL(P_m\square P_n) \ge D_\square+1 - 2\left\lfloor\log_2(m+1)\right\rfloor - \left\lfloor \frac{n-m+2}{m+1}\right\rfloor$.
\end{lemma}

\begin{proof}
Note that $D_\square=\diam(P_m\square P_n) = n+m-2$ and let $V(P_k) = [k]$. For notational convenience, we define $x=\left\lfloor \log_2(m+1)\right\rfloor-1$ and $y=\left\lfloor \frac{n-m+2}{m+1}\right\rfloor$. 
We cool $P_m\square P_n$ while avoiding cooling the vertex $(m,n)$ for as long as possible. 

We now provide a sequence of sources $(s_{\alpha,i})$ broken down into $2x+y+2$ phases, each consisting of a number of rounds depending on the phase $\alpha$. We note that if $m \leq 2^{x+1}$, then phases $x+1$ and $x+y+2$ do not occur, but the total number of sources remains the same.
During the first $x$ phases, we cool $s_{\alpha,i}=(2^\alpha-i,2i-1)$ in the $i$-th manual step of phase $\alpha$ for $1 \le i \le 2^\alpha-1$, with $s_{1,1}=(1,1)$. 
Each $s_{\alpha,i}$ has distance $m+n-2^\alpha-i+1 = m+n-\alpha-r$ from $(m,n)$, where $r=2^\alpha-\alpha+i-1$ is the round in which $s_{\alpha,i}$ is cooled. Phase $x$ ends with $2^{x+1}-2-x$ sources cooled. 

During phase $x+1$, we cool $s_{x+1,i}=(2^{x+1}-i,2i-1)$ in the $i$-th manual step for $1 \le i \le m-2^{x+1}$, noting that $s_{x+1,i}$ has distance $m+n-2^{x+1}-i+1 = m+n-\alpha-r$ from $(m,n)$, where $r=2^{x+1}-x+i-2$ is the round in which $(2^{x+1}-i,2i-1)$ is cooled. Phase $x+1$ ends with $m-x-2$ sources cooled. 

During phase $x+1+\alpha$ for $1 \le \alpha \leq y$, we cool $s_{x+1+\alpha,i}=(m+1-i,2i-1+(m+1)(\alpha-1))$ in the $i$-th manual step for $1 \le i \le m$ and note that $s_{x+1+\alpha,i}$ has distance $m+n-m\alpha-\alpha-i+1 = m+n-(x+1+\alpha)-r$ from $(m,n)$, where $r=m\alpha-x-2+i$ is the round in which $s_{x+1+\alpha,i}$ is cooled. Phase $x+y+1$ ends with $m(y+1)-x-2$ sources cooled. 

During phase $x+y+2$, we cool $s_{x+y+2,i}=(2m-2^{x+2}-i+1,n-2m+2^{x+2}+2i+2)$ in the $i$-th manual step for $1 \le i \le m-2^{x+1}$, noting that $s_{x+y+2,i}$ has distance $m-i-3 = m+n-(x+y+2)-r$ from $(m,n)$, where $r = n-x-y-1+i$ is the round in which $s_{x+y+2,i}$ is cooled. Phase $x+y+2$ ends with $n+m-x-y-2^{x+1}-1$ sources cooled. 

For the remaining $x$ phases, we cool $s_{x+y+2+\alpha,i}=(m+1-i,n-2^{x+2-\alpha}+2i+2)$ in the $i$-th manual step of phase $x+y+2+\alpha$ for $1 \le i \le 2^{x+1-\alpha}-1$, noting that $s_{x+y+2+\alpha,i}$ has distance $2^{x+2-\alpha}-i-3 = m+n-(x+y+2+\alpha)-r$ from $(m,n)$, where $r = m+n-(x+y+\alpha)-2^{x+2-\alpha}+i-1$ is the round in which $s_{x+y+2+\alpha,i}$ is cooled. Phase $2x+y+2$ ends with $n+m-2x-y-3 = D_\square-2x-y-1$ sources cooled. 

We now show that $(s_{\alpha,i})$ satisfies the cooling condition. 
For $s_{\alpha,i}$ and $s_{\alpha,j},$ note that  $d(s_{\alpha,i},s_{\alpha,j}) = 3|j-i|> |j-1|+1.$ 
The source cooled in round $r$ during phase $\alpha$ has distance $D_\square+2-\alpha-r = m+n-\alpha-r$ to $(m,n)$. Therefore, for $s_{\alpha,i}$ and $s_{\beta,j}$, then $d(s_{\alpha,i},(m,n)) \leq d(s_{\alpha,i},s_{\beta,j})+d(s_{\beta,j},(m,n))$ by the triangle inequality.
This yields $m+n-i-\alpha \leq d(s_{\alpha,i},s_{\beta,j})+m+n-j-\beta$, which implies that $$d(s_{\alpha,i},s_{\beta,j})\geq j-i+\beta-\alpha\geq |j-i|+1.$$
It follows that $(s_{\alpha,i})$ satisfies the cooling condition and cooling lasts until round $D_\square-2x-y-1$.
\end{proof}

Given Lemmas~\ref{lem:CART:isometric_subgraph_cartesian} and \ref{lem:CART:cartesian_rectangular_grids}, the proof of Theorem~\ref{thm:CART:primary_cartesian_bounds} now follows. Below, we provide computational results for $\CL(P_m \Box P_n)$ in Table~\ref{tab:CART:cartesian_product_table}, noting that the bound given by Lemma~\ref{lem:CART:cartesian_rectangular_grids} is within an additive factor of 3 of each result.

\begin{table}[h]
       \centering
       \begin{tabular}{|c|c|c|c|c|c|c|c|c|c|}
       \hline
           $\square$ & $P_2$ & $P_3$ & $P_4$ & $P_5$ & $P_6$ & $P_7$ & $P_8$ & $P_9$ & $P_{10}$ \\
           \hline
           $P_2$ & 2 & 3 & 4 & 4 & 5 & 6 & 6 & 7 & 8\\
           \hline
           $P_3$ & ~ & 4 & 4 & 5 & 6 & 7 & 7 & 8 & 9\\
           \hline
           $P_4$ & ~ & ~ & 5 & 6 & 7 & 7 & 8 & 9 & 10\\
           \hline
           $P_5$ & ~ & ~ & ~ & 7 & 7 & 8 & 9 & 10 & 11\\
           \hline
           $P_6$ & ~ & ~ & ~ & ~ & 8 & 9 & 10 & 11 & 12\\
           \hline
           $P_7$ & ~ & ~ & ~ & ~ & ~ & 10 & 11 & 12 & 12\\
           \hline
           $P_8$ & ~ & ~ & ~ & ~ & ~ & ~ & 11 & $\geq12$ & $\geq13$\\
           \hline           
           \end{tabular}
       \caption{The cooling number of $P_m \square P_n$.}
       \label{tab:CART:cartesian_product_table}
   \end{table}

We now have good lower bounds for the cooling number of Cartesian products based on the diameters of the factors, and so turn our attention to lower bounds given purely by the cooling numbers of the factors.
A graph is \emph{cooling critical} if there is precisely one uncooled vertex at the start of the last round during each maximum length cooling process on $G$. For an example of such a graph, consider $P_{2n}$.

\begin{theorem}
\label{thm:CART:cartesian_lower_by_cooling}
Let $G$ and $H$ be graphs. 
If both $G$ and $H$ are cooling critical, then
$\CL(G\square H) \geq \CL(G) + \CL(H) - 2$. Otherwise, 
$\CL(G\square H) \geq \CL(G) + \CL(H) - 1$.
\end{theorem}

\begin{proof}
Let $(\vA_1,\vA_2,\ldots, \vA_\mA)$ and $(\vB_1,\vB_2,\ldots, \vB_c)$ be optimal cooling sequences for $G$ and $H$, respectively.   
We cool sources in $G\square H$ as follows. In round $1\leq i \leq c$, we manually cool $x_i = (\vA_1,\vB_i)$. In round $c+1\leq j\leq c+b-1$, we manually cool $x_j = (\vA_{j-c+1},\vB_\mB)$. 
 
Since $d_{G\square H}((u_1,v_i),(u_1,v_j)) = d_H(v_i,v_j)$, the first $\mB$ vertices in the given sequence pairwise satisfy the cooling condition. Similarly, the final $\mA-1$ sources in the sequence pairwise satisfy the cooling condition because $$d_{G\square H}((u_{j-\mB+1},v_\mB),(u_{i-\mB+1},v_\mB)) = d_G(u_{j-\mB+1},u_{i-\mB+1}).$$
For $1\leq i \leq \mB$ and $1+\mB\leq j \leq \mA+\mB-1$, we have that
\begin{align*}
d_{G\square H}(x_i,x_j) 
&= d_{G\square H}((\vA_1,\vB_i),(\vA_{j-\mB+1},\vB_{\mB})) = d_G(\vA_1,\vA_{j-\mB+1}) + d_H(\vB_i,\vB_{\mB})\\
&\geq (|(j-\mB+1)-1|+1) + (|i-\mB|+1)\\
& > |j-i|+1.
\end{align*}

Therefore, $(x_1,x_2,\dots,x_{\mA+\mB-1})$ satisfies the cooling condition in $G\square H$ and we have that $\CL(G\square H) \geq b+c-1$. 
If $\CL(G)=\mA$ and $\CL(H)=\mB$, then we are done. 
    
If $\CL(G)=\mA+1$ and $\CL(H)=\mB$, then some $\vA_\infty\in V(G)$ is cooled in round $\mA+1$. Hence, $d_G(\vA_i,\vA_\infty)\geq \mA+1-i$ for $1 \leq i \leq \mA$, so $(\vA_\infty,\vB_{\mB})$ has distance at least $\mA+\mB-i$ from $x_i$ for $1\leq i \leq \mA+\mB-1$. Hence, $(\vA_\infty,\vB_{\mB})$ is not cooled by the end of round $\mA+\mB-1$ and cooling lasts $\mA+\mB = \CL(G)+\CL(H)-1$ rounds. An analogous argument shows that $\mA+\mB = \CL(G)+\CL(H)-1$ when 
$\CL(G)=\mA$ and $\CL(H)=\mB+1$. 

If $\CL(G)=\mA+1$ and $\CL(H)=\mB+1$, then some $\vA_\infty \in V(G)$ is cooled in round $\mA+1$ and some $\vB_\infty\in V(H)$ is cooled in round $\mB+1$. 
These have distance $d_G(\vA_i,\vA_\infty)\geq \mA+1-i$ from each $\vA_i$ for $1 \leq i \leq \mA$ and $d_H(\vB_i,\vB_\infty)\geq \mB+1-i$ from each $\vB_i$ for $1 \leq i \leq \mB$.
Therefore, $(\vA_\infty,\vB_\infty)$ has distance at least $\mA+\mB+1-i$ from each $x_i$ for $1\leq i \leq \mA+\mB-1$, implying that $(\vA_\infty,\vB_\infty)$ is not cooled by the propagation step of round $\mA+\mB$, and so $\CL(G\square H) \ge \mA+\mB = \CL(G)+\CL(H)-2$. 

Finally, without loss of generality, suppose $G$ is not cooling critical and that there exists some $w_\infty \in V(G)$ that satisfies the same condition as $\vA_\infty$. Hence, $(w_\infty,\vB_\infty)$ has distance at least $\mA+\mB+1-i$ from each $x_i$ for $1\leq i \leq \mA+\mB-1$, so there are two uncooled vertices after the propagation step of round $\mA+\mB$.
Only one of these can be manually cooled during round $\mA+\mB$, so cooling lasts until round $\mA+\mB+1 = \CL(G)+\CL(H)-1$.   
\end{proof}

The first bound of Theorem~\ref{thm:CART:cartesian_lower_by_cooling} is tight when $G=P_2$ and $H=C_3$. If instead $H=P_2$, then the second bound is tight. However, the first bound follows from a slightly stronger condition---in both factors, we assume that the last source in an optimal cooling sequence could be adjacent to the last vertex to cool. Thus, we ``run out of room'' and cooling is forced to end a round earlier. If this is not the case, then we can slightly improve the lower bound, even if both factors are cooling critical.

\begin{theorem} 
\label{thm:CART:cartesian_lower_by_cooling_w_critical}
Let $G$ and $H$ be cooling critical graphs. 
If there exists an optimal cooling sequence for $G$ where the last source is nonadjacent to the last cooled vertex, then
\[\CL(G\square H) \geq \CL(G) + \CL(H) - 1.
\]
\end{theorem}
\begin{proof}
Let $(\vA_1,\vA_2,\ldots, \vA_\mA)$ and $(\vB_1,\vB_2,\ldots, \vB_{\mB})$ be optimal cooling sequences for $G$ and $H$, respectively. 
Let $\vA_\infty$ denote the unique vertex that is cooled in round $\CL(G)$ when cooling $G$, and define $\vB_\infty$ likewise for $H$. 
    
We cool sources in $G\square H$ as follows. In round $1\leq i \leq \mA$, we manually cool the vertex $x_i = (\vA_i,\vB_1)$, noting that $d_{G\square H}((\vA_i,\vB_1),(\vA_j,\vB_1)) = d_G(\vA_i,\vA_j)$. 

In round $\mA+1$, we manually cool $x_{\mA+1} = (\vA_\mA,\vB_2)$, noting that $$d_{G\square H}((\vA_i,\vB_1),(\vA_\mA,\vB_2)) = d_G(\vA_i,\vA_\mA)+d_H(\vB_1,\vB_2) > (\mA+1)-i+1.$$
    
Finally, we manually cool $x_{\mA+j} = (u_\infty,v_j)$ in the manual step of rounds $\mA+j$ for $2\le j \le \mB$. For $1\leq i \leq \mA$, we have $d_G(\vA_i,\vA_\infty) \geq \mA -i+1$ and, because $\vA_\infty$ is the last source cooled in $G$, we have $d_G(\vA_\mA,\vA_\infty)\geq 2$. It is then straightforward to verify that $d_{G\square H}(x_i,x_{\mA+j})\geq (\mA+j)-i+1$, for $1\leq i \leq \mA+1$ and $2\leq j \leq \mB$. 

Therefore, $(x_i)_{i=1}^{b+c}$ satisfies the cooling condition. Finally, $d_{G\square H}(x_i,(\vA_\infty,\vB_\infty))\geq (b+c)-i+1$ for $1\leq i\leq \mA+\mB$, showing that $(\vA_\infty,\vB_\infty)$ is uncooled at the end of round $\mA+\mB$ and that cooling must last at least $\mA+\mB+1=\CL(G)+\CL(H)-1$ rounds. 
\end{proof}

Theorem~\ref{thm:CART:cartesian_lower_by_cooling_w_critical} is tight when $G$ and $H$ are full centipedes, and this hints toward our next development. So far, we have analyzed the cooling number of Cartesian product graphs through the lens of particularly chosen subgraphs. In Theorem~\ref{thm:CART:primary_cartesian_bounds}, we used the largest isometric subpath in each factor graph to find a lower bound for the cooling number of the Cartesian product.
However, if the factor graphs contain different types of isometric subgraphs, we may obtain tighter bounds. 

Trees are a natural first candidate for such isometric subgraphs, and Theorem~\ref{thm:CART:cartesian_lower_by_cooling_w_critical} tells us that maximally cool trees are a particularly good choice, as their Cartesian products are maximally cool. To see why, note that in a maximally cool tree of order at least $3$, the last source cannot be adjacent to the last cooled vertex, as both must be diameter distance from the first source.

\begin{corollary}\label{cor:CART:long_burning_tree_products}
Let $T_1$ and $T_2$ be maximally cool trees, each of order at least $3$. We then have that $T_1\square T_2$ is maximally cool. 
\end{corollary}

Therefore, if $G$ and $H$ each have subtrees that are both isometric and maximally cool, we know that $G \Box H$ must also be maximally cool. Denote by $\mathcal{T}_d$ the set of all maximally cool trees of diameter $d$. 

\begin{corollary}\label{cor:CART:isometric_subtrees}
    Let $G$ and $H$ be graphs. If $G$ and $H$ contain isometric subtrees in $\mathcal{T}_{\diam(G)}$ and $\mathcal{T}_{\diam(H)}$, respectively, then $G\square H$ is maximally cool. 
\end{corollary}

Armed with this new tool, we turn our attention to identifying particular classes of maximally cool trees. 

\begin{theorem}[{\cite[Theorem 3]{cooling}}] \label{thm:CART:cool_centipedes}
All full centipedes are maximally cool.
\end{theorem}

Using Theorem~\ref{lem:GENE:isometric_subgraph} and Theorem~\ref{thm:CART:cool_centipedes}, if a tree $T$ has a full centipede of length $k$ as a subtree, then $\CL(T) \geq k.$ In particular, this means that trees whose non-leaf vertices have high minimum degree are maximally cool, as they must contain an induced full centipede of diameter length.

\begin{corollary}
\label{cor:CART:cool_centipedes}
All trees with no vertices of degree 2 are maximally cool. 
\end{corollary}

However, these are not the only maximally cool trees. Inspired by the construction of full centipedes, we can construct an infinite family of maximally cool trees containing a vertex of degree 2. 

\begin{lemma}\label{lem:CART:cooling_centipedes_with_lobster_claws}
Let $T$ be the tree consisting of a path $u_1,u_2,\ldots,u_{2k+1},$ a leaf $x_i$ adjacent to $u_i$ for each $2 \leq i \leq 2k$ with $i \neq k+1,$ and leaves $y_1$ and $y_2$ adjacent to $x_k$. We then have that $T$ is maximally cool.
\end{lemma}

\begin{proof}
We use the cooling sequence $(\vA_1,x_2,x_3,\dots,x_{k-1},y_1,y_2,x_{k+2},x_{k+3},\dots,x_{2k})$.  
This sequence contains $2k=\diam(T)$ vertices. It is straightforward to show that this is a valid cooling sequence and that $u_{2k+1}$ is uncooled at the end of round $2k$. Therefore, $\CL(T) \geq 2k+1$. 
\end{proof}

Many graphs do not contain full centipedes as isometric subgraphs. However, all graphs contain a partial centipede as an isometric subgraph (for example, any diameter-length path), and the cooling number of these trees is known. 

\begin{lemma}[{\cite[Theorem 3.1.2]{hughes_thesis}}] \label{lem:CART:centipede_cooling}
If $T$ is a centipede with a spine of $n$ vertices and with $p$ spine vertices having a leaf, then
$
\CL(T) = \left\lfloor \frac{n+p+2}{2}\right\rfloor
$.
\end{lemma}

We can now start identifying centipedes as subgraphs to obtain bounds on more complex graphs. As an example, we can quickly find the cooling number of a well-known family of trees using this method.

\begin{theorem}\label{thm:CART:cool_regular_trees}
Let $d,h>1$. 
The complete $d$-ary tree $T$ of height $h$ is maximally cool, unless $d=h=2$, in which case $\CL(T) = \diam(T) = 2h$.  
\end{theorem}
\begin{proof}
If $d>2$, then any path from one leaf to another that contains the root is the spine of an induced full centipede. Since any subgraph of a tree is an isometric subgraph, the result for this case follows from Lemma~\ref{lem:GENE:isometric_subgraph} and Theorem~\ref{thm:CART:cool_centipedes} since $\diam(T)=2h$.  

If $d=2$ and $h=2$, then $T$ is a centipede with a spine of five vertices and with two leaves, and so Lemma~\ref{lem:CART:centipede_cooling} yields a cooling number of $4=\diam(T)$.

If $d=2$ and $h>2$, then we consider the subgraph containing a path $\vA_1,\vA_2,\ldots, \vA_{2h+1}$, along with a vertex $x_i$ adjacent to $\vA_i$ for each $2 \leq i \leq 2k$ and $i\neq k+1$, and with vertices $y_1$ and $y_2$ adjacent to $x_{k}$, noting that $\vA_{h+1}$ is the root of $T$ and $\vA_1$ and $\vA_{2k+1}$ are leaves.
By Lemmas~\ref{lem:GENE:isometric_subgraph} and \ref{lem:CART:cooling_centipedes_with_lobster_claws}, $T$ is maximally cool. 
\end{proof}

However, it is not just maximally cool trees whose Cartesian product becomes maximally cool. An \emph{alternating centipede} $T$ is a centipede with an odd number of spine vertices and a leg on every second spine vertex. 
By Lemma~\ref{lem:CART:centipede_cooling}, such a $T$ with $2k+1$ spine vertices has $\CL(T)=\left\lfloor \frac{3}{2}(k+1)\right\rfloor=\left\lfloor \frac{3}{4}(\diam(T)+2)\right\rfloor < \diam(T) +1$. 
However, the Cartesian product of two such trees is maximally cool. 

\begin{lemma}\label{lem:CART:alternating_centipedes}
If $G$ and $H$ are alternating centipedes, then $G\square H$ is maximally cool.
\end{lemma}

\begin{proof} 
Let the spine vertices of $G$ be $\vA_1,\vA_2,\ldots, \vA_{2a+1}$ and its leaves be $x_{2i}$ for $1\le i \le a$, and let the spine vertices of $H$ be $\vB_1,\vB_2,\ldots, \vB_{2b+1}$ and its leaves $y_{2i}$ for $1\le i \le b$. 

We construct a cooling sequence of $2a+2b$ sources as follows. First, let $\vC_1=(\vA_1,\vB_1)$. 
For $1\leq i \leq a$, let $\vC_{2i}=(x_{2i},\vB_1)$ and $\vC_{2i+1}=(x_{2i},y_2)$. Let $\vC_{2a+2} =(\vA_{2a+1},y_2)$.
Finally, for $1\leq i \leq b-1$, we let $\vC_{2a+2i+1}=(x_{2a},y_{2i+2})$ and $\vC_{2a+2i+2}=(\vA_{2a+1},y_{2i+2})$. This sequence has a total of $2a+2b = \diam(G\square H)$ vertices. 

Noting that $d_G(\vA_1,x_{2i})=2i$, that $d_G(x_{2i},x_{2j})=2+2|j-i|$ when $i\neq j$, and that $d_G(x_{2i},\vA_{2a+1})=2(a+1-i)$ for $1\leq i\leq a$, with distances in $H$ calculated analogously, it is straightforward to show that $d_{G\square H}(s_i,s_j) \geq |j-i|+1$ for $1\leq i,j\leq 2a+2b$.

Finally, we note that $d_{G\square H}(s_i,(\vA_{2a+1},\vB_{2b+1}))\geq 2a+2b-i+1$ for $1\leq i \leq 2a+2b$. Therefore, the vertex $(\vA_{2a+1},\vB_{2b+1})$ is uncooled at the end of round $2a+2b$ and cools in round $2a+2b+1 = \diam(G\square H)+1$.
\end{proof}

Lemma \ref{lem:CART:alternating_centipedes} shows that Corollary~\ref{cor:CART:isometric_subtrees} does not provide a full characterization of maximally cool Cartesian product graphs. Indeed, product graphs can have close to maximal cooling number even if their factor graphs contain only trivial isometric subtrees.

Consider the complete graph $K_2$. It contains only $P_2$ as an isometric subtree, and the product $\mBox_{i=1}^nK_2$ contains only $P_n$ as an isometric subtree. Theorem~\ref{thm:CART:primary_cartesian_bounds} yields that $\CL(\mBox_{i=1}^nK_2) \geq n+1-2\log_2(\left\lfloor\frac{n}{2}\right\rfloor+2).$ However, $\CL(\mBox_{i=1}^n K_{2}) = n$ by \cite[Theorem 3]{liminal}. Instead of relying on the subgraph arguments used above, we can use other structural properties of Cartesian products of complete graphs to study their cooling numbers. 

\begin{theorem}\label{thm:CART:cartesian_multiproduct_completegraphs}
Let $d_1,d_2,\ldots,d_n \geq 2$ be integers. If $n\geq 4$, then
$$
\CL(\mBox_{i=1}^n K_{d_i}) = \begin{cases}
    n & \text{if }d_1,d_2,\ldots,d_{n-1} \leq 2\text{ and }d_n \in \{2,3,4\},\\
    n+1 & \text{otherwise}
\end{cases}
$$
and if $n\in \{2,3\}$, then
$$
\CL(\mBox_{i=1}^n K_{d_i}) = \begin{cases}
    n & \text{if }d_1,d_2,\ldots,d_{n-1} \leq 2\text{ and }d_n \in \{2,3\},\\
    n+1 & \text{otherwise.}
\end{cases}
$$
\end{theorem}

We prove this result in three lemmas, each addressing different restrictions on the $d_i$. For notational convenience, we let $Q_n = \mBox_{i=1}^nK_2$. We label each vertex of $\mBox_{i=1}^n K_{d_i}$ with a word of length $n$, where the $i^{th}$ digit may take on values in $\{0,1,\dots,d_i-1\}$, such that two vertices are adjacent if and only if their corresponding labels differ in exactly one position. We write $(\overline{m}_i)$ to represent a word of length $i$ in which every digit is $m$, omitting the subscript when the length is clear from context. 
A single digit $m$ is written as $(m)$. Finally, we write $(\overline{m}_i)(\overline{\ell}_j)$ to denote the concatenation of words. We begin our classification of cliques with those whose products are not maximally cool.

\begin{lemma}
If $d \in \{3,4\}$ and $n \geq 4$, then
$
\CL(Q_n \square K_d) = \diam(Q_n \square K_d) = n+1.
$
\end{lemma}

\begin{proof}
We prove the theorem only for $d=4$, as the proof for $d=3$ is analogous. Let $G=Q_n \square K_4$, and, toward a contradiction, assume that $\CL(G) = n+2.$ Consider an optimal cooling sequence $(\vC_1,\vC_2,\ldots,\vC_k)$, noting that since $G$ is vertex-transitive, we can let $\vC_1=(\overline{0})$ without loss of generality. As $\diam(G) = n+1$, there can be no source chosen in round $n+2$, so the length of any cooling sequence for $G$ is at most $n+1$. This also implies that $\CL(G) = n+2$ only if there are at least two uncooled vertices of distance $n+1$ from $\vC_1$ after the propagation step of round $n+1$. 
    
Therefore, at least two of $\{(\overline{1}), (\overline{1}_n)(2),(\overline{1}_n)(3)\}$ are uncooled after the propagation step of round $n+1$. This implies that $\vC_n \not \in \{(\overline{1}_n)(0),(\overline{1}), (\overline{1}_n)(2),(\overline{1}_n)(3)\}$, as otherwise all of these vertices would be cooled in round $n+1$, so $\vC_n$ must be a word of distance $n$ from $\vC_1$ ending in 1, 2, or 3. Without loss of generality, we assume that $\vC_n$ ends in 1, implying that $(\overline{1})$ is cooled during the propagation step in round $n+1$. In order to avoid $(\overline{1}_n)(2)$ and $(\overline{1}_n)(3)$ also being cooled during this propagation step, none of the $\vC_i$ may end in either $2$ or $3$. To see why, note that each $\vC_i$ has distance at least $|i-1|+1=i$ from $\vC_1$ and therefore has at least $i$ nonzero digits. If $\vC_i$ ends in 2, then it shares at least $i$ digits with $(\overline{1}_n)(2)$ and at least $i-1$ digits with $(\overline{1}_n)(3)$. Hence, $d(\vC_i,(\overline{1}_n)(2)) \leq n+1-i < |(n+1)-i|+1$, so $(\overline{1}_n)(2)$ cannot be a source in round $n+1$ or uncooled at the end of round $n+1$. A similar argument shows that no $\vC_i$ can end in 3. Therefore, every $\vC_i$ must end in 0 or 1.  

It must also be true that no $\vC_i$ ends in $0$. If so, then it would share exactly $i$ nonzero digits in common with both $(\overline{1}_n)(2)$ and $(\overline{1}_n)(3)$. Neither is a valid source in round $n+1$, so each $\vC_i$ must end in $1$. This also implies that each $\vC_i$ contains exactly $i$ nonzero digits, including the final $1$. Now, consider $\vC_n;$ it must be $1$ in every digit except some digit $j$, where it is $0$. Therefore, $\vC_2$ is zero everywhere except in digits $j$ and $n+1$, and so every other $\vC_i$ must not contain a 1 as digit $j$. However, this implies that $\vC_3$ is a binary word with $j^{th}$ digit is a 0, which has exactly three nonzero digits, and which shares at most three digits with $\vC_n$, which is a contradiction. Therefore, $\CL(G) \leq n+1.$ 

To derive equality, we note that the $(n+1)$-dimensional hypercube $Q_{n+1}$ is an isometric subgraph of $G$. By Lemma~\ref{lem:GENE:isometric_subgraph}, $\CL(G) \geq \CL(Q_{n+1})=n+1$. 
\end{proof}

In \cite{liminal}, a key tool for determining the cooling number of the hypercube $Q_n$ is a type of Sperner family. We say that a family $\mathcal{F}$ of subsets of $[n]$ is a $k$-\textit{rainbow Sperner family of cardinality} $j$ if the following conditions are satisfied:
\begin{enumerate}
\item $|\{S \in \mathcal{F}~:~|S| = i\}| = k$ for all $2 \leq i \leq j+1$.
\item For all $S,T \in \mathcal{F}$, $S \not\subseteq T$. 
\end{enumerate}
In the case where $k=1$, we say that $\mathcal{F}$ is a \textit{rainbow Sperner family of cardinality} $j$. Rainbow Sperner families of cardinality $n-2$ always exist and always satisfy the cooling condition in $Q_n$, as shown in \cite{liminal}. 

\begin{lemma}[{\cite[Lemma 1]{liminal}}]\label{lem:rainbow_sperner_exists}
    For each $n\geq 3$, there exists a rainbow Sperner family on $[n]$ of cardinality $n-2$. 
\end{lemma}

\begin{lemma}[{\cite[Lemma 2]{liminal}}]\label{lem:rainbow_sperner_symdif}
    Let $S_2,S_3,\ldots,S_{n-1}$ be a rainbow Sperner family of cardinality $n-2$. For all $2 \leq i\neq j \leq n-1,$ we have $|S_i \Delta S_j| \geq |i-j|+2$. 
\end{lemma}

Now, we show that such a family of sets can be used to construct a cooling sequence for more general Cartesian products of complete graphs. We begin by showing that if at least two factors of $\mBox_{i=1}^n K_{d_i}$ are not $K_2$, then the product is maximally cool.

\begin{lemma}
If $n \geq 4$, $d_1,d_2,\ldots,d_n \geq 2$ are integers, and at least two of the $d_i$ are at least $3$, then 
$
\CL(\mBox_{i=1}^n K_{d_i}) = \diam(\mBox_{i=1}^n K_{d_i})+1 = n+1
$.
\end{lemma}

\begin{proof}
Without loss of generality, let $d_{n-1},d_n\geq 3$. We construct a sequence of sources for which cooling lasts $n+1$ rounds. Let $\{S_2,S_3,\ldots,S_{n-2}\}$ be a rainbow Sperner family on $[n-1]$ satisfying the following conditions:
\begin{enumerate}
    \item $|S_i| = i$, for all $2 \leq i \leq n-2$,
    \item $n-1 \in S_i$, for all $2 \leq i \leq n-3$, and 
    \item $S_{n-2} = [n-2].$
\end{enumerate} 
Such a family exists by Lemma~\ref{lem:rainbow_sperner_exists}. Each $S_i$ corresponds to a binary word given by $w^{(i)} = w_1^{(i)}w_2^{(i)}\cdots w_{n-1}^{(i)}$ where
$$
w_j^{(i)} = \begin{cases}
1 & \text{if } j \in S_i , \\
0 & \text{if }j \not \in S_i .
\end{cases}
$$
We define a cooling sequence $(\vC_1,\vC_2,\ldots,\vC_{n})$ as follows. Let $\vC_1 = (\overline{0})$, $\vC_2=(\overline{0}_{n-2})(2)(1)$, and $\vC_i = w^{(i-1)}(2)$ for all $3 \leq i \leq n-1$. By Lemma~\ref{lem:rainbow_sperner_symdif}, $d(\vC_i,\vC_j) \geq |i-j|+1$ for $3 \leq i < j \leq n-1$. For $3\leq i \leq n-1$, we have that $d(\vC_1,\vC_i) = i \geq |i-1|+1$ and $d(\vC_2,\vC_i) = i \geq |i-2|+1$. Finally, as $d(\vC_1,\vC_2) = 2$, we see that $(\vC_1,\vC_2,\ldots,\vC_{n-1})$ is a valid sequence of sources. 

Now, we identify a valid source in round $n$ and a vertex that is uncooled at the end of round $n$. 
We have that $d((\overline{1}),\vC_1) = n$, that $d((\overline{1}),\vC_2)= n-1 = |n-2|+1,$ and that $d((\overline{1}),\vC_i) = |n-i| +1$, for $3\leq i\leq n-1$. Therefore, $\vC_{n} = (\overline{1})$ is a valid source in round $n$ and we let $\vC_n = (\overline{1})$. Now, consider the vertex $s_\infty = (\overline{1}_{n-2})(\overline{2}_2)$. A similar argument, using the condition that $n-1\in S_i$ for all $2\leq i \leq n-3$, shows that $d(\vC_\infty,s_i) \geq |n-i|+1$ for $1\leq i \leq n-1$. Finally, we have that $d(\vC_n,\vC_\infty)=2$. Therefore, $\vC_\infty$ is uncooled at the end of round $n$, and we have that $\CL(\mBox_{i=1}^n K_{d_i}) = n+1.$
\end{proof}

Our final result below shows that if even one of the factors of $\mBox_{i=1}^n K_{d_i}$ is large enough, then the product is maximally cool.

\begin{lemma}
If $n \geq 3$ and $m \geq 5$, then 
$
\CL(Q_n \square K_m) = \diam(Q_n \square K_m)+1 = n+2.
$
\end{lemma}

\begin{proof}
We construct a sequence of sources for which cooling lasts $n+2$ rounds. Let $\{S_2,S_3,\ldots,S_{n-1}\}$ be a rainbow Sperner family on $[n]$  satisfying the following conditions:
\begin{enumerate}
    \item $|S_i| = i$ for all $2 \leq i \leq n-1$,
    \item $n \in S_j$ for all $2 \leq j \leq n-2$, and 
    \item $S_{n-1} = [n-1].$
\end{enumerate}  
Such a family exists by Lemma~\ref{lem:rainbow_sperner_exists}. Each $S_i$ corresponds to a binary word $w^{(i)} = w_1^{(i)}w_2^{(i)}\cdots w_n^{(i)},$ where
$$
w_j^{(i)} = \begin{cases}
1 & \text{if }j \in S_i, \\
0 & \text{if }j \not \in S_i.
\end{cases}
$$

We define a cooling sequence $(\vC_1,\vC_2,\ldots,\vC_{n})$ as follows. Let $\vC_1 = (\overline{0})$, $\vC_2=(\overline{0}_{n-1})(1_2)$, and let $\vC_i = w^{(i-1)}(2)$ for all $3 \leq i \leq n$. By Lemma~\ref{lem:rainbow_sperner_symdif}, $d(\vC_i,\vC_j) \geq |i-j|+1$ for $3 \leq i < j \leq n$. For $3\leq i \leq n$, we have that $d(\vC_1,\vC_i) = i = |i-1|+1$ and $d(\vC_2,\vC_i) = i-1 = |i-2|+1$. Finally, as $d(\vC_1,\vC_2) = 2$, we see that $(\vC_1,\vC_2,\ldots,\vC_n)$ is a valid sequence of sources. 

Now, we identify a valid source in round $n+1$ and a vertex that is uncooled at the end of round $n+1$. 
We have that $d((\overline{1}_n)(3),\vC_1) = n+1$, $d((\overline{1}_n)(3),\vC_2)= n \geq |n+1-2|+1,$ and that $d((\overline{1}_n)(3),\vC_i) = n+2-i \geq |n+1-i|+1.$  
Therefore, $\vC_{n+1}=(\overline{1}_n)(3)$ is a valid source in round $n+1$. A similar argument shows that $d((\overline{1}_n)(4),\vC_i)\geq |n+1-i|+1$ for $1\leq i \leq n$. Finally, we have that $d((\overline{1}_n)(4),\vC_{n+1}) = 1$.  Therefore, $\vC_\infty = (\overline{1}_n)(4)$ is uncooled at the end of round $n+1$, and we have that $\CL(Q_n \square K_m) = n+2.$
\end{proof}

Finally, by calculating that $\CL(K_3 \square K_3) = 3$, that $\CL(K_2\square K_3 \square K_3) = 4$, and that $\CL(Q_2 \square K_5)=4$ and applying Lemma~\ref{lem:GENE:isometric_subgraph}, we obtain the result in Theorem~\ref{thm:CART:cartesian_multiproduct_completegraphs}.

We have now seen that factor graphs with only the trivial $P_2$ as isometric subtrees can result in product graphs that are maximally cool. In fact, these are exactly the diameter length subpaths we considered in Theorem~\ref{thm:CART:primary_cartesian_bounds}. 
Next, we consider multi-dimensional Cartesian grids $\mBox_{i=1}^{n}P_k$. We label each vertex of $\mBox_{i=1}^{n}P_k$ with a word of length $n$, where the $i^{th}$ digit may take on values in $\{0,1,\dots,k\}$. In this case, labels are assigned in such a way that adjacent vertices receive labels $u$ and $v$ satisfying $\sum_{i=1}^{n}|(u)_i-(v)_i|=1$, where $(u)_i$ denotes the $i^{th}$ digit of the word $u$. It follows that the distance between vertices corresponding to words $u$ and $v$ is given by $d(u,v) = \sum_{i=1}^{n}|(u)_i-(v)_i|$. 

Now, by applying Theorem~\ref{thm:CART:primary_cartesian_bounds} with $G=\mBox_{i=1}^{\lfloor n/2\rfloor}P_k$ and $H=\mBox_{i=1}^{\lceil n/2\rceil}P_k$, we are able to derive the following lower bound.

\begin{corollary}\label{thm:CART:other_cartesian_path_product}
If $k \geq 3$ and $n \geq 2$ are integers, then 
\[\CL(\mBox_{i=1}^{n}P_k) \geq (k-1)n-1-2\left\lfloor\log_2\left((k-1)\left\lfloor \frac{n}{2}\right\rfloor+2\right)-1\right\rfloor.\]   
\end{corollary}

However, there is once again an emergent structure that we can exploit if $n$ is much larger than $k$. Indeed, for $n =\Omega(2^k)$, we can improve the bound of Corollary~\ref{thm:CART:other_cartesian_path_product}. 

\begin{theorem}\label{thm:CART:improved_cartesian_path_product}
If $k \geq 3$ and $n \geq 2$, then 
$\CL(\mBox_{i=1}^{n}P_k) \geq (k-1)(n-1)+1$.
\end{theorem}

\begin{proof}
We construct a sequence of sources $(\vC_1,\vC_2,\ldots,\vC_{(k-1)(n-1)+1})$ of length $(k-1)(n-1)+1$ in $G=\mBox_{i=1}^{n}P_k$ as follows.
Let $\vC_1 = (\overline{0})$ and, for $2 \leq i \leq n$, let 
$$
\vC_i = \begin{cases}
(\overline{0}_{n+1-i})(2)(\overline{1}_{i-2}) & \text{if } 2 \leq i \leq n-1, \\
(\overline{1}) & \text{if }i = n.
\end{cases}
$$
When $i = n+(n-1)x+y$ for $0 \leq x \leq k-3$ and $1 \leq y \leq n-1$, we define $s_i$ as 
$$
(s_i)_j = \begin{cases}
x+2 & \text{if } 1 \leq j \leq y+1, \\
x & \text{if }j = y+2, \\
x+1 & \text{if }y+3 \leq j \leq n.
\end{cases}
$$
We now show that each of these is a valid source by computing the distance between pairs of sources. We consider each of the following five cases separately.
    
\begin{enumerate}
\item  $d(\vC_n,\vC_i)$ for $1 \leq i \leq n-1$. 
        
Here, we have $d(\vC_n,\vC_i) = \sum_{j=1}^{n}|(\vC_i)_j-1| \geq \sum_{j=1}^{n+1-i}1 = n-i+1 \geq |n-i|+1,$ as required.\\

\item  $d(\vC_n,\vC_i)$ for $i = n+(n-1)x+y.$ 
        
As is required, we have that \begin{align*}
d(\vC_n,\vC_i) &= \sum_{j=1}^{n}|(\vC_i)_j-1| = \sum_{j=1}^{y+1}(x+1) + |x-1| + \sum_{j=1}^{n-y-2}x\\
&= (x+1)(y+1) +|x-1|+x(n-y-2)  \\
&= (n-1)x+y+1+|x-1| \\
&\geq |n+(n-1)x+y-n|+1.
\end{align*}
        
\item $d(\vC_i,\vC_\ell)$ for $1 \leq \ell < i \leq n-1.$
        
We have $d(\vC_i,\vC_\ell) = \sum_{j=1}^{n}|(\vC_i)_j-(\vC_\ell)_j| \geq |2-0|+\sum_{j=1}^{i-\ell-1}|1-0| = |i-\ell|+1,$ as required.\\

\item  $d(\vC_i,\vC_\ell)$ for $1 \leq i \leq n-1$ and $\ell = n+(n-1)x+y$.
        
Let us assume that $ 1<i < n-y$. Here, we have that 
\[
d(\vC_i,\vC_\ell) = \sum_{j=1}^n|(\vC_i)_j -(\vC_\ell)_j| = \sum_{j=1}^{y+1}(x+2)+x+\sum_{j=1}^{n-1-i-y}(x+1)+|x-1|+\sum_{j=1}^{i-2}x .
\]
If $0<n-y-i$, then
\begin{align*}
d(\vC_i,\vC_\ell)&= (x+2)(y+1)+x+(x+1)(n-1-i-y)+|x-1|+(i-2)x\\
&=(n-1)x+y+1+|x-1|+n-i\\
&\geq |n+(n-1)x+y-i|+1.
\end{align*}
If, instead, $0=n-y-i$, then 
\begin{align*}
d(\vC_i,\vC_\ell)&= (x+2)(y+1)+x+|x-1|+(i-2)x\\
&=(y+i)x+2y+2x+|x-1|\\
&=nx+y+n-i+2x+|x-1|\\
&\geq |n+(n-1)x+y-i|+1,
\end{align*}
as required. Similar calculations show that $d(\vC_i,\vC_\ell)\geq |n+(n-1)x+y-i|+1$ in the case that $i \geq n-y$ or $i=1$.\\
        
\item  $d(\vC_i,\vC_\ell)$ for $i = n+(n-1)x+y$ and $\ell = n+(n-1)x'+y'.$
        
We have that $(\vC_i)_j = x+2-(\vC_{n-y})_j$ and $(\vC_\ell)_j = x'+2-(\vC_{n-y'})_j,$ for all $1\leq j\leq n$. We then have that
\begin{align*}
d(\vC_i,\vC_\ell) &= \sum_{j=1}^{n}\left|\left(x+2-(\vC_{n-y})_j\right)-\left(x'+2-(\vC_{n-y'})_j\right)\right| \\
&= |x-x'|n+d(\vC_{n-y},\vC_{n-y'}) \\
&\geq |x-x'|n+|y-y'|+1\\
&\geq |n+(n-1)x+y - (n+(n-1)x'+y')|+1,
\end{align*}
where the third line is from Case 3.
\end{enumerate}
    
We have therefore shown that the sequence $(\vC_1,\vC_2,\ldots,\vC_{(k-1)(n-1)+1})$ is a valid cooling sequence, and so $\CL(\mBox_i^nP_k) \geq (k-1)(n-1)+1$. 
\end{proof}

The use of Theorem~\ref{thm:CART:improved_cartesian_path_product} is broader than it initially seems. Given a collection of graphs $G_1,G_2\ldots,G_n$ with minimum diameter $k$, Theorem~\ref{thm:CART:improved_cartesian_path_product} yields the bound $\CL(\mBox_{i=1}^nG_i) \geq (n-1)(k-1)+1.$ While this bound is general, it is often not tight. The graphs whose cooling numbers fall within this gap are reminders that isometry is not the only effective tool in our toolbox, and suggest that graph cooling is primed for deeper, more structural study. However, if we treat $k$ as a constant with respect to $n$, then the smaller $k$ is, the tighter the bound. Indeed, for $k=3$, Theorem~\ref{thm:CART:improved_cartesian_path_product} is tight.

\begin{theorem}\label{thm:CART:p3_cartesian_product_exact}
If $n \geq 3$, then $\CL(\mBox_{i=1}^nP_3) = 2n-1.$
\end{theorem}

\begin{proof}
Let $G = \mBox_{i=1}^nP_3$. Theorem~\ref{thm:CART:improved_cartesian_path_product} yields that $\CL(G) \geq 2n-1$, so we need only show that $\CL(G) \leq 2n-1$. Any diameter length path in $G $ begins at a word $w$ composed of only $0$ and $2$ and ends at $(\overline{2})-w,$ where subtraction is done component-wise. Thus, $|N_{2n}[v]\setminus N_{2n-1}[v]|\leq 1$ for every $v\in V(G)$ and $\CL(G) \leq \diam(G) = 2n$ by Theorem~\ref{thm:GENE:eccentricity}.

Toward a contradiction, suppose that $\CL(G) = 2n$ and let $(\vC_i)$ be an optimal cooling sequence for $G$. 
The first source, $\vC_1$, must be the endpoint of a diameter length path in $G$. Without loss of generality, assume that $\vC_1=(\overline{0}).$  The last vertex to cool must be $(\overline{2})$. We analyze both cases for the last round of cooling.

\begin{enumerate}
\item Suppose that $(\overline{2})$ is cooled manually in round $2n$. In particular, $\vC_{2n}=(\overline{2})$. For $2\leq i\leq 2n$, we have that $d(\vC_1,\vC_i) \geq |i-1|+1 = i$, and for $1\leq i \leq 2n-1$, we have $\vC_{2n} = (\overline{2})$. However, $$d(\vC_{2n},\vC_i) = \sum_{j=1}^{n} |(\vC_{2n})_j-(\vC_i)_j| = \sum_{j=1}^{n} | 2-((\vC_i)_j-(\vC_1)_j)|=  2n-d(\vC_1,\vC_i) \leq 2n-i,$$ for $2 \leq i \leq 2n-1,$ a contradiction. Therefore, $\vC_{2n} \neq (\overline{2})$.\\

\item Suppose that $(\overline{2})$ is cooled by propagation in round $2n.$ We have that $$d(\vC_1,\vC_{2n-1}) \geq 2n-1+1 = 2n-1,$$ so $\vC_{2n-1} = (\overline{2})-(1)_i$, where $(1)_i$ has $i^{th}$ digit 1 and 0 elsewhere. Therefore, $(\vC_i)_j = 2$ for all $2 \leq i \leq 2n-2$, because if $(\vC_i)_j = 0$, then $$d(\vC_i,\vC_{2n-1}) = 1+(2n-2)-i = 2n-1-i < |2n-1-i|+1,$$ failing the cooling condition. Likewise, if $(\vC_i)_j = 1$, then $$d(\vC_i,\vC_{2n-1}) = 2n-2-(i-1) = 2n-1-i < |2n-1-i|+1.$$ However, this implies that $(\vC_2)_j=(\vC_3)_j=2,$ and as $d(\vC_2,\vC_1) = 2$ and $d(\vC_3,\vC_1) =3$, this forces $d(\vC_2,\vC_3) = 1$, a contradiction. Hence, $(\overline{2})$ was not cooled by propagation in round $2n$.
\end{enumerate}
    
As both possibilities lead to a contradiction, $\CL(G) < 2n,$ as desired. 
\end{proof}

However, Theorem~\ref{thm:CART:improved_cartesian_path_product} is not tight for all $k\geq 4$. Note that $\CL(\mBox_{i=1}^{3} P_4) = 8,$ while Theorem~\ref{thm:CART:improved_cartesian_path_product} gives $\CL(\mBox_{i=1}^{3} P_4) \geq 7.$

Finally, we handle the combination of these two cases: the Cartesian product of a complete graph and a path. By Theorem~\ref{thm:CART:primary_cartesian_bounds}, we have that $\CL(P_m \square K_n)\geq \lceil \frac{2}{3}m\rceil -1=\diam(P_m \square K_n)-1-\lfloor\frac{m}{3}\rfloor$. However, if $n > m \geq 2$, then we can improve this and show that $P_m \square K_n$ is maximally cool.

\begin{theorem}
For $m>1$ and any $n$, we have that
$$\CL(P_m \square K_n) = \diam(P_m \square K_n) +1-\left\lfloor \frac{m+1}{n+1}\right\rfloor.$$

\end{theorem}

\begin{proof}
For the lower bound, we start by cooling any vertex $\vC_1$ corresponding to a leaf in $P_m$. We then cool subsequent sources by cooling any valid vertex in the copy of $K_n$ that is not completely cooled and is closest to $\vC_1$. In this way, $m+1$ vertices are cooled each round except the first and last, for a total of $\left\lceil \frac{mn-1}{n+1}\right\rceil$ additional rounds. 

For the upper bound, note that in round 1, exactly one vertex $(\vA_1,\vB_1)$ is cooled. 
During the propagation step in round 2, the $n-1$ vertices of the form $(\vA_1,v)$ for $v\in V(K_n)\setminus\{\vB_1\}$ are cooled, along with at least one vertex of the form $(u,\vB_1)$ with $u \neq \vA_1$. 
Since, for all $v\in V(K_n)$, there is a cooled vertex of the form $(u,v)$, in each propagation step (except for the last) there will be an uncooled vertex $(u',v)$ that is cooled by propagation in the following round.

That is, during each propagation step, except in the last round, at least $n$ vertices are cooled, and during the first and last rounds, at least $n+1$ uncooled vertices are cooled. 
There are $mn-1$ uncooled vertices at the start of round $2$, and so after an additional $\left\lceil \frac{mn-1}{n+1}\right\rceil$ rounds, all of these vertices must be cooled. This completes the proof once we note that $\left\lceil \frac{mn-1}{n+1}\right\rceil+1=(m+1)+\left\lceil \frac{-m-1}{n+1}\right\rceil=\diam(P_m\square K_n)+1-\left\lfloor\frac{m+1}{n+1}\right\rfloor$. 
\end{proof}

\section{Cooling strong products}

We begin this section by comparing the cooling number of the strong product of two graphs with the Cartesian product of the same pair of graphs. 

\begin{theorem} \label{thm:STRONG:strong_lowerbounds_cartesian}
For any graphs $G$ and $H$, we have $\CL(G \boxtimes H) \leq \CL(G \square H)$. 
\end{theorem}

\begin{proof}
For $(\vA_1,\vB_1),(\vA_2,\vB_2) \in V(G) \times V(H)$, we have that $$d_{G\boxtimes H}((\vA_1,\vB_1),(\vA_2,\vB_2)) \leq d_{G\square H}((\vA_1,\vB_1),(\vA_2,\vB_2)).$$
Therefore, if $(\vC_1,\vC_2,\dots,\vC_\mC)$ is a cooling sequence for $G\boxtimes H$, by Lemma~\ref{lem:GENE:cooling_sequence_equation}, we have that $|j-i|+1\leq d_{G\boxtimes H}(\vC_i,\vC_j) \leq d_{G\square H}(\vC_i,\vC_j)$, and so $(\vC_1,\vC_2,\dots,\vC_\mC)$ is also a valid cooling sequence for $G \square H$. Finally, if there is some vertex $s_\infty$ of $G\boxtimes H$ that is uncooled at the end of round $e$ using the cooling sequence above, then $$ e-i+1 \leq d_{G\boxtimes H}(\vC_\infty,\vC_i)\leq d_{G\square H}(\vC_\infty,\vC_i).$$ So, $\vC_\infty$ is also uncooled at the end of round $e$ in $G\square H$, and therefore $\CL(G\square H) \geq \CL(G\boxtimes H)$. 
\end{proof}
      
This bound is rarely tight, because the diameter of $G\boxtimes H$ and $G \square H$ are typically very different. 
Although we could use the results from the previous section along with Theorem~\ref{thm:STRONG:strong_lowerbounds_cartesian}, we can achieve much better results by studying the strong product directly. 
Define $D_\boxtimes = \diam(G\boxtimes H) = \max(\diam(G),\diam(H)).$

\begin{theorem}\label{thm:STRONG:strong_general_bounds}
Let $G$ and $H$ be graphs with $\diam(G) \leq \diam(H)$. If $\diam(H) \le 2\diam(G)-1,$ then $G\boxtimes H$ is maximally cool.  If instead $\diam(H) \geq 2\diam(G)$, then
\[
 D_\boxtimes+1 - \left\lceil \frac{D_\boxtimes - 2\diam(G)+3}{\overline{\CL}(G)+1}\right\rceil \leq \CL(G\boxtimes H)\leq D_\boxtimes+1,
\]
where $\overline{\CL}(G)$ is the length of a longest optimal cooling sequence for $G$. 
\end{theorem}

\begin{proof}
The upper bound follows from Theorem~\ref{thm:GENE:diam_bounds}.
To show the lower bound, we provide an example of a cooling sequence of the desired length. 
Let $m=\diam(G)+1$, $n=\diam(H)+1$, and suppose we label the vertices so that $(1,2,\ldots,m) \subseteq V(G)$ is 
an induced path of $G$ of length $\diam(G)$, and $(1,2,\ldots,n) \subseteq V(H)$ is an induced path of $H$ of length $\diam(H)$. Let $(g_1, g_2, \ldots, g_c)$ be an optimal cooling sequence in $G$. We cool $G\boxtimes H$ in three phases by providing a sequence of sources $(b_i).$ 

{\bf Phase 1: }
This phase lasts until the end of the propagation step of round $m$. 
In round $i$ with $1 \leq i \leq \left\lceil\frac{m}{2}\right\rceil$, we cool $b_i = (2i-1,1)$. 
For $1\le i < j \le \left\lceil\frac{m}{2}\right\rceil$, we note that $d_{G\boxtimes H}(b_i,b_j) = 2|j-i|>|j-i|+1$, and so these $(b_i)$ satisfy the cooling condition.
For the remainder of the first phase, we split into cases based on the parity of $m$. 

If $m = 2k$ is even, we cool $b_{i+k} = (m,2i+1)$ in round $i+k$, for $1 \leq i \leq k-1$. 
Note that $d_{G\boxtimes H}(b_{i+k},b_{j+k})=2|j-i|\geq |j-i|+1$ and $$d_{G\boxtimes H}(b_i,b_{j+k}) =\max\{m-2i+1,2j\}\geq |(j+k)-i|+1.$$ This last inequality holds because
$$
\max\{m-2i+1,2j\}\geq \frac{(m-2i+1)+(2j)}{2}= j+k-i+\frac{1}{2},
$$
and the distance must be an integer, so we may take the ceiling. Thus, these $(b_i)$ satisfy the cooling condition.

If $m = 2k+1$ is odd, then we cool $b_{i+k+1} = (m,2i+2)$ in round $i+k+1$ for the values $1 \leq i \leq k-1$. 
Note that $d_{G\boxtimes H}(b_{i+k+1},b_{j+k+1})=2(j-i)\geq |j-i|+1$ and $d_{G\boxtimes H}(b_i,b_{j+k+1})=\max\{m-2i+1,2j+1\}\geq |(j+k+1)-i|+1$. This last inequality holds because 
$$
\max\{m-2i+1,2j+1\}\geq \frac{(m-2i+1)+(2j+1)}{2} \geq (j+k+1)-i+\frac{1}{2},
$$
and the distance must be an integer, so we may take the ceiling. Thus, these $(b_i)$ satisfy the cooling condition.

We have played $m-1$ rounds in both the odd and even cases. 
During the propagation step of round $m$, all vertices of distance $m-1 = \diam(G)$ from $b_1=(1,1)$ will be cooled, and no vertex $(i,j)$ with $j>m$ has been cooled. 
That is, the set of cooled vertices is exactly those vertices $(u,v)$ with $d_H(1,v)\le m-1$.

{\bf Phase 2:} 
If $\diam(H) \le 2\diam(G)-1,$ we skip phase 2, noting that in this case, vertices of distance $N'=m-1$ from $b_1$ were cooled by the propagation step of round $I'=m$ during the first phase.  
Otherwise, we play phase 2 as a sequence of subphases, with subphase $s$ starting after the propagation step of round $r_s=m+(s-1)c$. We take as our inductive hypothesis that at the start of subphase $s$, the vertices of distance $D_s=m-1+(s-1)(c+1)$ from $b_1$ are cooled, 
and that at the end of subphase $s$, the vertices of distance $D_{s+1}=m-1+s(c+1)$ from $b_1$ are cooled 
for the values $1 \leq s \leq \left\lceil (D_\boxtimes-2m+3)/(c+1)\right\rceil$. 

In round $i$ of subphase $s$, we cool $b_{r_s-1+i} = (g_i, D_s +1 + i)$. 
Once the propagation step of round $i+1$ is performed, 
all vertices of distance up to $D_s +i$ from $b_1$ are cooled, 
the vertices of distance $D_s +i+2$ or more from $b_1$ are uncooled, and the vertices $(u,D_s+i+1)$ that are cooled are distance at most $i-j+1$ from $(g_j,D_s+i+1)$ for $1 \leq j \leq i$.

The set of $u\in V(G)$ such that $(u,D_s+i+1)$ are cooled are those where $u\in V(G)$ would be cooled when cooling the sequence $(g_i)$ on $G$, so $(g_{i+1},D_s+i+1)$ satisfies the cooling condition for $i < c$. 
As a result, during the propagation step of round $r_{s+1}$, all $(u,D_{s+1})$ are cooled for $u\in V(G)$. 
Though there may be uncooled vertices $(u,v)$ with $d_{G\boxtimes H}((u,v),b_1)=D_{s+1}$, we assume that these are cooled, as this only speeds up the cooling. Thus, our induction holds true.

At the end of phase $2$, the vertices of distance 
$$
N'=m-1+\left\lceil \frac{D_\boxtimes-2m+3}{c+1}\right\rceil (c+1)\geq D_{\boxtimes}-m+2
$$ 
from $b_1$ have been cooled by the end of the propagation step of round $I'$, where we have that
$
I' = m+\left\lceil \frac{D_\boxtimes-2m+3}{c+1}\right\rceil c
$.

{\bf Phase 3: }
Define $n'=n-N'-1$, which we note is the maximum value so that the vertices of distance at most $n'-1$ from $(m,n)$ remain uncooled. 
In round $I'+i-1$ for $1 \leq i \leq \left\lceil\frac{n'}{2}\right\rceil$, we 
manually cool the source  
$(1,N'+2i)$.  
If $n'$ is odd, we manually cool $(1+2i,n)$ in round $I'+\left\lceil \frac{n'}{2}\right\rceil+i-1$ for $1 \leq i \leq \left\lfloor\frac{n'}{2}\right\rfloor$, noting that $1+2\left\lfloor\frac{n'}{2}\right\rfloor=n'\leq m-1$ so each of these manually cooled vertices are vertices of the grid. 
If $n'$ is even, then we manually cool $(2i,n)$ 
in round $I'+\left\lceil \frac{n'}{2}\right\rceil+i$ for $1 \leq i \leq \left\lfloor\frac{n'}{2}\right\rfloor$, 
noting that $2\left\lfloor\frac{n'}{2}\right\rfloor=n'\leq m-1$ so each of these manually cooled vertices are vertices of the grid.
In either case, we cool $n'$ vertices in this phase, with the final vertex being $(y,n)$ with $y=n'\leq m-1$, meaning $(m,n)$ is uncooled.
As such, cooling lasts until round $I'+n'$, and the result follows in the case $\diam(G) \leq \diam(H) \le 2\diam(G)-1$ since $I'+n'=n = D_\boxtimes+1$, and it follows in the case $\diam(H) \ge 2\diam(G)$ since 
$$
I'+n' 
= \left(m+\left\lceil 
\frac{D_\boxtimes-2m+3}{c+1}
\right\rceil c\right)+ (n-N'-1) 
= D_\boxtimes +1 - \left\lceil 
\frac{D_\boxtimes-2m+3}{c+1}
\right\rceil.
\qedhere$$
\end{proof}

If both graphs are paths, the upper bound of Theorem~\ref{thm:STRONG:strong_general_bounds} can be strengthened using a result about the isoperimetry of $P_m \boxtimes P_n$. 

\begin{theorem}[\cite{strong_product_isoperimetry_paper}] \label{thm:STRONG:strong_grid_isoperimetry}
Let $m\geq n$ and $S\subseteq V(P_m \boxtimes P_n)$ such that $$\frac{m^2-2m+1}{4} \leq |S| \leq mn-\frac{m^2+2m+5}{4}.$$ 
If $N(S) = \{v\in N(u) : u \in S, v \notin S\}$, then 
\[
|N(S)|\geq 
\begin{cases}
m & \text{if } m \text{ divides } |S|,\\
m+1 & \text{otherwise.}
\end{cases}
\]
\end{theorem}

If $G$ and $H$ are paths, then the upper and lower bounds can be shown to differ only by an additive factor of $\left\lceil \frac{m}{2} \right\rceil$. 
This is an improvement over Theorem~\ref{thm:STRONG:strong_general_bounds} in the case that $n > \frac{m^2}{4}$.

To derive the next result, we introduce the concept of \emph{waves} of play during the cooling process. 
Suppose the first cooled vertex is $u$ and the last cooled vertex we aim to cool is $v$. 
If round $r$ occurs during wave $w$, then the only vertices that may be manually cooled during round $r$ are those of distance at least $d(u,v)+2-w-r$ from $v$. 
Note that this means that no vertex of distance at most $d(u,v)+1-w-r$ from $v$ will be cooled during this round.  
As a natural consequence, the cooling number will be at least $d(u,v)+2-w$, presuming there is at least one round in wave $w$. 

\begin{theorem}\label{thm:STRONG:strong_grid_result}
Let $4 \leq m \leq n=D_\boxtimes+1$. 
If $n \leq 2m-2$, then $P_m \boxtimes P_n$ is maximally cool. 
If instead $n \geq 2m-1$, then
\[
 D_\boxtimes+1 - \left\lceil \frac{D_\boxtimes - 2m+5}{\left\lceil\frac{m}{2}\right\rceil+1}\right\rceil \le \CL(P_m \boxtimes P_n) 
 \le 
 D_\boxtimes+1- \left\lceil\frac{D_\boxtimes - 2m+3}{\left\lceil\frac{m}{2}\right\rceil+1}\right\rceil+\left\lceil\frac{m}{2}\right\rceil .
\]
\end{theorem}

\begin{proof}
The case for $n\leq 2m-2$ is done by Theorem~\ref{thm:STRONG:strong_general_bounds}, and so we assume for the remainder of the proof that $n \geq 2m-1$.
The lower bound is exactly Theorem~\ref{thm:STRONG:strong_general_bounds} when $G=P_m$ and $H=P_n$, so it remains to show the upper bound. Define $\numsteps =\left\lceil (D_\boxtimes - 2m+3)/(\left\lceil\frac{m}{2}\right\rceil+1)\right\rceil$. Suppose that $(x,y)$ was cooled in the first round and call $C_j=\{(i,j): 1 \leq i\leq m\}$ column $j$ for $j=1,2,\ldots,n$. 
During the propagation step of round $m$, all $C_{y+m'}$ are cooled for $-(m-1)\leq m' \leq m-1$, so at least $m$ columns are cooled, implying at least $m^2$ vertices are cooled. 
Note that this implies we must at least be in wave 2. 

We now describe the cooling process over waves $2$ to $\numsteps+1$. 
Wave $w\geq 2$ starts on or before round $m+(w-2) \left\lceil \frac{m}{2}\right\rceil$, and we assume that at least a total of $c_w = m\left(m+(w-2) (\left\lceil \frac{m}{2}\right\rceil+1)\right)$ vertices are cooled after the propagation step of the first round contained in wave $w$, but before a vertex is manually cooled. 
We will show that this holds by induction on $w$, with the base case $w=2$ already established.
To perform the induction, we will show that at least $c_{w,i} = c_w +(i-1)(m+2)$ vertices are cooled after the propagation step of the $i$th round of wave $w$ for $1 \leq i \leq \left\lceil \frac{m}{2}\right\rceil$, which itself will be shown using an induction.
This is true for $i=1$ by the definition of $c_{w,1}$. 
For the inner induction, assume the statement is true for case $i$, and we will show it is true for the case $i+1$. 
We manually cool a vertex to yield at least $c_{w,i}+1$ cooled vertices, which completes the $i$th round of wave $w$.

We note that during each propagation step of these waves, we have between $m^2 \geq \frac{m^2-2m+1}{4}$ and $m\left(D_\boxtimes -m+3 \right) \leq mn - \frac{m^2+2m+5}{4}$ cooled vertices, where the latter inequality holds since $m\geq 4$.
As such, the conditions on $|S|$ of Theorem~\ref{thm:STRONG:strong_grid_isoperimetry} holds when we take $S$ to be the set of cooled vertices during round $i$ of each wave $2$ to $W+1$ where $1 \leq i \leq \left\lceil \frac{m}{2}\right\rceil$.

Let $\cv_{i+1}$ be the set of cooled vertices at the start of round $i+1$ of wave $w$, and so $|\cv_{i+1}|\geq c_{w,i}+1$ when $i\geq 1$. 
If $m$ divides $|\cv_{i+1}|$, then $|\cv_{i+1}| \geq c_{w,i}+2$, since $m$ does not divide $c_{w,i}+1$ for $1\leq i<\left\lceil \frac{m}{2}\right\rceil$. 
Theorem~\ref{thm:STRONG:strong_grid_isoperimetry} yields that $|N(\cv_{i+1})| \geq m$, so at least $m$ additional vertices will be cooled during the propagation step for a total of at least 
$c_{w,i}+m+2 = c_{w,i+1}$ cooled vertices at the end of the propagation of the $(i+1)$th round of wave $w$, as we needed to show. 

If $m$ does not divide $|\cv_{i+1}|$, then Theorem~\ref{thm:STRONG:strong_grid_isoperimetry} yields that $|N(\cv_{i+1})| \geq m+1$, so the number of cooled vertices after the propagation step of round $i+1$ of wave $w$ is at least $c_{w,i}+1+(m+1) = c_{w,i+1}$. 
This completes the induction, and so our claim that there are at least $c_{w,i} = c_w +(i-1)(m+2)$ cooled vertices after the propagation step of the $i$th round of wave $w$ holds. 

It still remains to be shown that there are at least $c_{w+1}$ cooled vertices after the propagation step of the first round of wave $w+1$. 
At the end of round $\lceil\frac{m}{2}\rceil$ of wave $w$, there were at least $c_w+(\lceil \frac{m}{2}\rceil-1)(m+2)+1$ cooled vertices. 
If $m$ is odd, there are at least $c_w+(\lceil \frac{m}{2}\rceil-1)(m+2)+m+1= c_w+m(\lceil \frac{m}{2}\rceil+1) = c_{w+1}$ cooled vertices.
If $m$ is even, it could be that there were exactly $c_w+(\lceil \frac{m}{2}\rceil-1)(m+2)+1$ cooled vertices, which is not divisible by $m$, so there are at least $c_w+(\lceil \frac{m}{2}\rceil-1)(m+2)+1+m+1= c_w+m(\lceil \frac{m}{2}\rceil+1) = c_{w+1}$ cooled vertices after the propagation step. Otherwise there were at least $c_w+(\lceil \frac{m}{2}\rceil-1)(m+2)+2$, and so after the propagation step there are at least $c_w+(\left\lceil \frac{m}{2}\right\rceil-1)(m+2)+2+m= c_w+m(\left\lceil \frac{m}{2}\right\rceil+1) = c_{w+1}$ cooled vertices.
This completes the induction. 

By following this process, we find ourselves at the end of round  $\lceil \frac{m}{2}\rceil$ of wave $W+1$. 
We allow another propagation round, in which the conditions of Theorem~\ref{thm:STRONG:strong_grid_isoperimetry} also hold.
After this propagation step, the total number of cooled vertices is at least:$$m\left(m+W\left(\left\lceil \tfrac{m}{2}\right\rceil+1\right)\right) 
\ge 
m\left(m+\left(\frac{D_\boxtimes - 2m+3}{\left\lceil\frac{m}{2}\right\rceil+1} \right) \left(\left\lceil \tfrac{m}{2}\right\rceil+1\right)\right)=m\left(D_\boxtimes -m+3 \right),$$ which has occurred over 
$m+\left\lceil\frac{m}{2}\right\rceil \numsteps
\le 
D_\boxtimes - m+3+\left\lceil\frac{m}{2}\right\rceil- \numsteps
$ rounds.
We claim that all uncooled vertices are within distance $m-2$ of these cooled vertices. To see why, note that if there exists an uncooled vertex $v$ such that $N_{m-2}(v)$ is all uncooled, then there must be at least $(m-1)^2$ uncooled vertices and at most $m(n+1-2m)$ cooled vertices in $P_m\boxtimes P_n$.  
Since $m> 3$, this contradicts the fact that there are at least $m\left(D_\boxtimes -m+3 \right)$ cooled vertices, so all uncooled vertices must be distance at most $m-2$ from some cooled vertex. Hence, all vertices will be cooled after $m-2$ rounds and so $\CL(P_m \boxtimes P_n) \leq D_\boxtimes+1- \left\lceil\left(D_\boxtimes - 2m+3\right)/\left(\left\lceil\frac{m}{2}\right\rceil+1\right)\right\rceil+\left\lceil\frac{m}{2}\right\rceil$.
\end{proof}

Note that when $m>n/2$, Theorem~\ref{thm:STRONG:strong_general_bounds} applies to show that $P_m\boxtimes P_n$ is maximally cool. This generalizes a result from \cite{ambrose} which applies only to square strong grids. 

We conclude this section with Table~2, which provides exact values of the cooling number for strong grids derived by computation.

   \begin{table}[htpb!]
\label{paths_strong_table}
       \centering
       \begin{tabular}{|c|c|c|c|c|c|c|c|c|c|}
       \hline
           $\boxtimes$ & $P_2$ & $P_3$ & $P_4$ & $P_5$ & $P_6$ & $P_7$ & $P_8$ & $P_9$ & $P_{10}$ \\
           \hline
           $P_2$ & 2 & 3 & 3 & 4 & 4 & 5 & 5 & 6 & 6\\
           \hline
           $P_3$ & ~ & 3 & 4 & 4 & 5 & 6 & 6 & 7 & 8\\
           \hline
           $P_4$ & ~ & ~ & 4 & 5 & 6 & 6 & 7 & 8 & 8\\
           \hline
           $P_5$ & ~ & ~ & ~ & 5 & 6 & 7 & 8 & 8 & 9\\
           \hline
           $P_6$ & ~ & ~ & ~ & ~ & 6 & 7 & 8 & 9 & 10\\
           \hline
           $P_7$ & ~ & ~ & ~ & ~ & ~ & 7 & 8 & 9 & 10\\
           \hline
           $P_8$ & ~ & ~ & ~ & ~ & ~ & ~ & 8 & 9 & $10$\\
           \hline           
           \end{tabular}
       \caption{The cooling number of $P_i \boxtimes P_j$.}
       \label{tab:STRONG:strong_product_table}
   \end{table}

\section{Cooling lexicographic products}

We next turn to the lexicographic product, the only noncommutative of the four graph products considered in this paper. The structure of $G\circ H$ may be thought of as follows: for each vertex $v$ in $G$, there is a corresponding copy of $H$, call it $H_v$. Moreover, if $uv\in E(G)$, then in $G\circ H$, every vertex in $H_v$ is adjacent to every vertex in $H_u$. Due to this structure, the cooling number of $G\circ H$ depends mostly on $G$.

\begin{theorem} 
\label{thm:LEX:circ_general}
If $G$ is a connected graph of order at least 2, and $H$ is any graph, then
$$\CL(G) \leq \CL(G \circ H) \leq \left\lfloor \frac{3}{2} \CL(G)  \right\rfloor+1.$$
\end{theorem}

\begin{proof}
Let $(\vA_{1},\vA_{2},\dots,\vA_{\mA})$ be an optimal cooling sequence for $G$.

For the lower bound, let $v\in V(H)$ and consider the subgraph of $G \circ H$ induced by $\{(u,v) : u\in V(G)\}$. This subgraph is isomorphic to $G$, and $d_{G\circ H}((\vA,v),(w,v)) = d_G(\vA,w)$. Note that $d_G(\vA_{i},\vA_{j}) > i-j$ for all $1\leq j <i\leq \mA$, so we also have the inequality $d_{G\circ H}((\vA_{i},v),(\vA_{j},v)) > i-j$. Thus, $((\vA_{1},v),(\vA_{2},v),\dots,(\vA_{\mA},v) )$ satisfies the cooling condition in $G\circ H$, so we have found a cooling sequence of length $\mA$ for $G\circ H$. 

If $\CL(G)=\mA$, then $\CL(G \circ H) \geq \mA=\CL(G)$ and we are done. Otherwise, $\CL(G)=\mA+1$ and there is some uncooled $\vA_{\infty} \in V(G)$ at the end of round $\mA$ using the cooling sequence for $G$ above. For $1\leq j\leq \mA$, we have $\mA-j<d_G(\vA_{j},\vA_\infty) = d_{G \circ H}((\vA_{j},v),(\vA_\infty,v))$ and $(\vA_\infty,v)$ is uncooled at the end of round $\mA$ using the given cooling sequence. Thus, $\CL(G\circ H) \geq \mA+1 = \CL(G)$. 

For the upper bound, let $S = \big((\vA_1,\vB_1),(\vA_2,\vB_2),\ldots, (\vA_\mC,\vB_\mC)\big)$ be an optimal cooling sequence for $G \circ H$. 
We say that the index $i$ is a \emph{repeat} if $\vA_{i-1}=\vA_i$. 

We will define a cooling sequence for $G$ using $S$. 
First, we decompose $S$ into subsequences as follows, beginning with $i=1$ and repeating while $i\leq e$. 
\begin{enumerate}
\item If $i+1$ and $i+3$ are both repeats, then let $B_i = \big((\vA_i,\vB_i),$ $(\vA_{i+1},\vB_{i+1}),$ $(\vA_{i+2},\vB_{i+2}),$ $(\vA_{i+3},\vB_{i+3})\big)$. Increment $i$ by $4$ and repeat.
\item If $i<\mC-1$ and $i+1$ is a repeat but $i+3$ is not, then let $B_i = \big((\vA_i,\vB_i),$ $(\vA_{i+1},\vB_{i+1}),$ $(\vA_{i+2},\vB_{i+2})\big)$. Increment $i$ by $3$ and repeat.  
\item If $i+1$ is not a repeat or $i+1=\mC$, then let $B_i = \big((\vA_i,\vB_i)\big)$. Increment $i$ by $1$ and repeat. 
\end{enumerate}
When $i>e$, then we have decomposed the given cooling sequence for $G\circ H$ into subsequences that can be concatenated to recover the original cooling sequence. 

We construct a cooling sequence for $G$ by processing each of the subsequences $B_i$, beginning with $B_1$, as follows. 
\begin{enumerate}
\item If $B_i = \big((\vA_i,\vB_i)\big)$, then we add $\vA_i$ as the next vertex in the cooling sequence for $G$, except in the case that $i+1=\mC$ and $i+1$ is a repeat, in which case we do not add a source to the cooling sequence for $G$. Increment $i$ by $1$ and repeat if $i<\mC$.
\item If $B_i = \big((\vA_i,\vB_i),(\vA_{i+1},\vB_{i+1}),(\vA_{i+2},\vB_{i+2})\big)$, then add $\vA_i$ and $\vA_{i+2}$ as the next two sources in the cooling sequence for $G$. Increment $i$ by $3$ and repeat if $i<\mC-2$. 
\item If $B_i = \big((\vA_i,\vB_i),$ $(\vA_{i+1},\vB_{i+1}),$ $(\vA_{i+2},\vB_{i+2}),$ $(\vA_{i+3},\vB_{i+3})\big)$, then add $\vA_{i}$, $x$, $\vA_{i+2}$ as the next three sources in the cooling sequence for $G$, where $x$ is any vertex such that a vertex of the form $(x,v)$ is cooled in the propagation step of round $i+2$ of the cooling process on $G\circ H$ using the given cooling sequence. Increment $i$ by $4$ and repeat if $i<\mC-3$. 
\end{enumerate}

Note that in each case, the number of sources added to the cooling sequence for each subsequence is at least $\frac{2}{3}$ of the number of vertices in the corresponding subsequence $B_i$, except perhaps for the last subsequence. 

Now, we show that the sequence constructed above is a valid cooling sequence for $G$. We note any $\vA_i$ and $\vA_j$ in the constructed sequence satisfy $\vA_i\neq \vA_j$. Moreover, they satisfy the cooling condition because $d_G(\vA_i,\vA_j) = d_{G\circ H}\big((\vA_i,\vB_i),(\vA_j,\vB_j)\big)$. Thus, we need only consider any vertex $x$ added in case (3). Suppose the vertex $x$ is added to the sequence from the subsequence $B_i$, so $(x,v)$ is cooled in the propagation step of round $i+2$ in $G\circ H$. For a vertex $u_j$ in the constructed sequence, we have that $$ d_G(u_j,x) = d_{G\circ H}\big((u_j,v_j),(x,v)\big) \geq \begin{cases}
        |i+2-j| & \text{if}~ j<i+2\\
        |i+2-j|+1 & \text{if}~j\geq i+2.
        \end{cases}$$
In particular, we get that $d_G(u_i,x) \geq 2$ and $d_G(u_{i+2},x) = d_G(u_{i+3},x) \geq 2$. 
The difference in position index in the constructed cooling sequence for $G$ between $x$ and $u_j$, with $j<i$, is at most $i+1-j$. Similarly, if $j>i+3$, the difference in position index is at most $|i+2-j|$. It follows that $u_j$ and $x$ satisfy the cooling condition in $G$. 

Finally, suppose that the vertex $x$ is added to the sequence from the subsequence $B_i$ and the vertex $x'$ from the subsequence $B_j$, with $i<j$. In $G\circ H$, the vertex $(x,v)$ cools by propagation in round $i+2$ and $(x',v')$ cools by propagation in round $j+2$. Thus, $$d_G(x,x') = d_{G\circ H}\big((x,v),(x',v')\big) \geq (j+2)-(i+2) = j-i.$$ The difference in position indices between $x$ and $x'$ is at most $j-i-1$, and so $x$ and $x'$ satisfy the cooling condition in $G$. Therefore, the sequence constructed above is a valid cooling sequence for $G$. 
    
The cooling sequence constructed above contains at least $\frac{2}{3}(\mC-1)$ sources. In the case that $\mC = \CL(G\circ H)$, then we have $\CL(G\circ H) \leq \frac{3}{2}\CL(G)+1$, and we are done. If instead $\mC = \CL(G\circ H)-1$, then there exists some $(u_\infty,v_\infty) \in V(G \circ H)$ such that $(u_\infty,v_\infty)$ is uncooled at the end of round $e$. We shall consider two subcases based on whether or not $u_e$ is a repeat, and show that the upper bound of the theorem holds in both cases.

 If $u_e$ is a repeat, then we must have $u_e \neq u_\infty$, as otherwise $(u_\infty,v_\infty)$ would be cooled at the end of round $e$ by propagation from $(u_{e-1},v_{e-1}).$ Because $(u_\infty,v_\infty)$ cools by propagation in round $e+1$, we must have  
    $
    d_{G \circ H}\big((u_\infty,v_\infty),(u_i,v_i)\big) \geq |e-i|+1
    $,
    which implies $d_G(u_\infty,u_i) \geq |e-i|+1$ for all $1 \leq i \leq e.$ Therefore, $u_\infty$ is not cooled by propagation at the end of round $e$ by any $u_i$ in the cooling sequence. Now, suppose that $x$ is added to the cooling sequence from the subsequence $B_i$, so $(x,v)$ is cooled in the propagation step of round $i+2$ in $G\circ H$. We must therefore have 
    $$
    d_{G\circ H}\big((u_\infty,v_\infty),(x,v)\big) \geq |e+1-(i+2)| = |e-(i+1)|,
    $$
    which implies $d_G(u_\infty,x) \geq |e-(i+1)|.$ Because $x$ is cooled in round $i+1$ of the cooling sequence or earlier, we therefore have that $u_\infty$ is not cooled by propagation from any $x$ in the cooling sequence. This finally implies that $u_\infty$ must be uncooled after each of the $\frac{2}{3}(e-1)$ sources in the original sequence are cooled, so cooling must last at least $\frac{2}{3}(e-1)+1$ rounds. Therefore, we have  $\CL(G) \geq \frac{2}{3}(\CL(G\circ H)-2)+1$, which implies $\CL(G \circ H) \leq \frac{3}{2}\CL(G)+\frac{1}{2} \leq \frac{3}{2}\CL(G)+1$.
    
If $u_e$ is not a repeat, then the cooling sequence contains at least $\frac{2}{3}e = \frac{2}{3}(\CL(G \circ H)-1)$ sources and we have that $\CL(G \circ H) \leq \frac{3}{2}\CL(G)+1$ as before. The cooling number must be an integer, so we take the floor to obtain the result. \end{proof}

The upper bound in Theorem \ref{thm:LEX:circ_general} appears to be within a constant additive factor of being tight. For example, consider the graph in Figure~\ref{fig:LEX:lex_upper_bound}. This graph, $G$, has $\CL(G) = 9$, and its lexicographic product with $\overline{K}_2$ has cooling number 13, one less than the upper bound given by Theorem \ref{thm:LEX:circ_general}. 

\begin{figure}[ht]
\centering
\begin{tikzpicture}[thick, every node/.style={circle, draw=black, fill=black, inner sep=1}]

\node (n1) at (0,0){};
\node (n2) at (1,0){};
\node (n3) at (2,0){};
\node (n4) at (3,0){};
\node (n5) at (4,0){};
\node (n6) at (5,0){};
\node (n7) at (6,0){};
\node (n8) at (7,0){};
\node (n9) at (8,0){};
\node (n10) at (9,0){};
\node (n11) at (10,0){};
\node (n12) at (11,0){};

\node (y1) at (0.5,1){};
\node (x3) at (2,1){};
\node (y4) at (3.5,1){};
\node (x6) at (5,1){};
\node (y7) at (6.5,1){};
\node (x9) at (8,1){};
\node (y10) at (9.5,1){};
\node (x12) at (11,1){};

\draw (n1) -- (n2);
\draw (n2) -- (n3);
\draw (n3) -- (n4);
\draw (n4) -- (n5);
\draw (n5) -- (n6);
\draw (n6) -- (n7);
\draw (n7) -- (n8);
\draw (n8) -- (n9);
\draw (n9) -- (n10);
\draw (n10) -- (n11);
\draw (n11) -- (n12);
\draw (y1) -- (n1);
\draw (y1) -- (n2);
\draw (x3) -- (n3);
\draw (y4) -- (n4);
\draw (y4) -- (n5);
\draw (x6) -- (n6);
\draw (y7) -- (n7);
\draw (y7) -- (n8);
\draw (x9) -- (n9);
\draw (y10) -- (n10);
\draw (y10) -- (n11);
\draw (x12) -- (n12);

\end{tikzpicture}
\caption{A graph $G$ with $\CL(G) = 9$ whose lexicographic product with $\overline{K}_2$ has cooling number $\CL(G \circ \overline{K}_2) = \left\lfloor \frac{3}{2}\CL(G)\right\rfloor  = 13$.}
\label{fig:LEX:lex_upper_bound}
\end{figure}
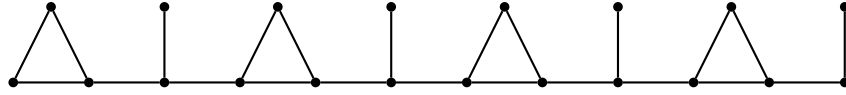

We note the difference between the cooling number of the lexicographic product $G\circ H$ and the burning number of the same graph, which satisfies $b(G\circ H) \in \{b(G),b(G)+1\}$, as shown in \cite{prod}. 

Next, we provide a general class of graphs attaining the lower bound in Theorem~\ref{thm:LEX:circ_general}. 

\begin{corollary}
\label{cor:LEX:circ_diam}
If $G$ is a connected, maximally cool graph with $\diam(G) \geq 2$, then $\CL(G \circ H) = \CL(G)$ for any graph $H$. 
\end{corollary}

\begin{proof}
We have $\diam(G \circ H) = \diam(G)$, since $\diam(G)\geq 2$. By Theorem~\ref{thm:GENE:diam_bounds}, we must have $\CL(G\circ H) \leq \diam(G\circ H)+1 = \CL(G)$. Equality follows from Theorem~\ref{thm:LEX:circ_general}. 
\end{proof}

Note that Corollary~\ref{cor:LEX:circ_diam} does not depend on $H$. In general, $\CL(G \circ H)$ depends only on $G$ and on whether $\CL(H) \leq 2$, provided that $\CL(G)$ is sufficiently large. 

\begin{theorem} \label{thm:LEX:H}
Let $G$ be a connected graph of order at least 2 and let $H_1$ and $H_2$ be graphs each containing a pair of nonadjacent vertices. We then have the following.
\begin{enumerate}
\item If $\CL(H_1),\CL(H_2) = 2$ or $\CL(H_1),\CL(H_2)\geq 3$, then $\CL(G \circ H_1) = \CL(G \circ H_2).$\\

\item If $\CL(H_1)=2$ and $\CL(H_2)\geq 3$, then $\CL(G \circ H_1) \leq \CL(G \circ H_2) \leq \CL(G \circ H_1)+1$. 
\end{enumerate}
\end{theorem}

\begin{proof}
Let $B_1 = \big((\vA_1,\vB_1),(\vA_2,\vB_2),\dots,(\vA_\mC,\vB_\mC)\big)$ be an optimal cooling sequence for the graph $G\circ H_1$. Note that $\mC\geq 2$. We have two cases. 

\smallskip

\noindent \emph{Case 1:} $\vA_1,\vA_2,\dots,\vA_\mC$ are distinct vertices of $G$. 

\smallskip

Let $w$ and $w'$ be the first and second sources in a cooling sequence for $H_2$. Since $H_2$ is not a clique, there exists a cooling sequence of length at least 2 for $H_2$. We show that $B_2 = \big((\vA_1,w),(\vA_2,w),\dots,(\vA_\mC,w)\big)$ 
satisfies the cooling condition in $G\circ H_2$ because 
$$
d_{G\circ H_2}((\vA_i,w),(\vA_j,w)) = d_G(\vA_i,\vA_j) = d_{G\circ H_1}((\vA_i,\vB_i),(\vA_j,\vB_j)) > |i-j|.$$ 

Thus, $\CL(G\circ H_2) \geq e$. If $\CL(G \circ H_1) = \mC$, then we are done. Otherwise, we have that $\CL(G \circ H_1)=\mC+1$ and there is some $(x,y)$ that cools in round $\mC+1$ due to $B_1$. Therefore, we have that 
$d_{G \circ H_1}((\vA_i,\vB_i),(x,y)) \geq \mC+1-i$ for $i=1,2,\dots,\mC$. There are three sub-cases for $x$. 

\smallskip

\emph{Case 1.a:} First, if $x\not\in \{\vA_1,\vA_2,\dots,\vA_\mC\}$ and $z\in V(H_2)$, then 
$$
d_{G\circ H_2}((x,z),(\vA_i,w)) = d_G(x,\vA_i) = d_{G\circ H_1}((x,y),(\vA_i,\vB_i)) \geq \mC+1-i,
$$
so $(x,z)$ is uncooled at the end of round $\mC$ with respect to $B_2$. 

Now suppose $x\in \{\vA_1,\vA_2,\dots,\vA_\mC\}$. Observe that every $(\vA_i,v)$ in $G\circ H_1$ is cooled after the propagation step of round $i+2$. Hence, $(x,y)$ is uncooled at the end of round $\mC$, so $x\in \{\vA_{\mC-1},\vA_\mC\}$. This yields the last two sub-cases for $x$.

\smallskip

\emph{Case 1.b:} If $x = \vA_{\mC-1}$, then 
$
d_{G\circ H_1}((\vA_{\mC-1},\vB_{\mC-1}),(x,y)) = d_{H_1}(\vB_{\mC-1},y) \geq 2
$.
In $G\circ H_2$, consider $(\vA_{\mC-1},w')$. Note that
$d_{G\circ H_2}((\vA_{\mC-1},w),(\vA_{\mC-1},w')) = d_{H_2}(w,w') \geq 2$, which means that $(x,w')$ is uncooled at the end of round $\mC$ with respect to $B_2$. 

\smallskip

\emph{Case 1.c:} Finally, if $x=\vA_\mC$, then we have that $d_{G \circ H_1}((\vA_\mC,\vB_\mC),(x,y)) = d_{H_1}(\vB_\mC,y) \geq 1$ and  $d_{G\circ H_2}((\vA_\mC,w),(\vA_\mC,w')) = d_{H_2}(w,w')\geq 2$, so $(x,w')$ is uncooled after round $\mC$ with respect to $B_2$ and $\CL(G\circ H_1) \leq \CL(G\circ H_2)$. Case~1 now follows.

\smallskip

\noindent\emph{Case 2:} Not all of $\vA_1,\vA_2,\dots,\vA_\mC$ are distinct. 

\smallskip

If $\vA_i = \vA_j$ with $i<j$, then $i+1=j$ is required to satisfy the cooling condition.

As before, let $w$ and $w'$ be the first two sources in a cooling sequence $B_2$ for $H_2$. For $i=1,2,\dots,\mC$, we let $$g(\vB_i) = \begin{cases}
w & \text{if~}\vA_i \neq \vA_{i-1}, \\
w' & \text{if~}\vA_i = \vA_{i-1}.
\end{cases}$$
We will show that $(\vA_i,g(\vB_i))$ satisfies the cooling condition in $G \circ H_2$ for each $i$.

If $|i-j| > 1$, then 
$
d_{G\circ H_2}((\vA_i,g(\vB_i)),(\vA_j,g(\vB_j))) = d_G(\vA_i,\vA_j)>|i-j|
$. 
If $j=i+1$ and $\vA_{i+1}\neq \vA_i$, then 
$
d_{G\circ H_2}((\vA_i,g(\vB_i)),(\vA_{i+1},g(\vB_{i+1}))) = d_G(\vA_i,\vA_{i+1}) > 1
$.
Finally, if $\vA_i = \vA_{i+1}$, then 
$
d_{G\circ H_2}((\vA_i,g(\vB_i)),(\vA_{i+1},g(\vB_{i+1}))) = d_{H_2}(g(\vB_i),g(\vB_{i+1})) = d_{H_2}(w,w')>1$, 
since $w,w'$ are consecutive sources in a cooling sequence for $H_2$. 

Thus, $\CL(G\circ H_2) \geq \mC$. If $\CL(G \circ H_1) = \mC$, then we are done. Otherwise, we have that $\CL(G \circ H_1)=\mC+1$ and there is some $(x,y)$ that cools in round $\mC+1$ due to $B_1$. Therefore, we have that
$
d_{G \circ H_1}((\vA_i,\vB_i),(x,y)) \geq \mC+1-i
$ for $i=1,2,\dots,\mC$. There are three subcases for $x$. 

\smallskip

\emph{Case 2.a:} First, suppose $x\not\in \{\vA_1,\vA_2,\dots,\vA_\mC\}$. If $z\in V(H_2)$, then 
$$
d_{G\circ H_2}((x,z),(\vA_i,g(\vB_i))) = d_G(x,\vA_i) = d_{G\circ H_1}((x,y),(\vA_i,\vB_i)) \geq \mC+1-i,
$$ 
so $(x,z)$ is not cooled at the end of round $\mC$ with respect to $B_2$. 

For the remaining two sub-cases, we must have $x\in \{\vA_{\mC-1},\vA_{\mC}\}$. The argument depends on whether or not $\vA_{\mC-1}$ and $\vA_{\mC}$ are in fact the same vertex in $G$. 

\smallskip

\emph{Case 2.b:} Suppose that $\vA_{\mC-1}\neq \vA_\mC$. If $x=\vA_{\mC-1}$, then $\vA_{\mC-2}\neq\vA_{\mC-1},$ implying that $g(\vB_{\mC-1}) = w$. Now, for $x\in \{\vA_{\mC-1},\vA_\mC\}$, the same argument as in \emph{Case 1} shows that $(x,w')$ is uncooled at the end of round $\mC$ with respect to $B_2$. Thus, $\CL(G\circ H_1) \leq \CL(G\circ H_2)$. 

\smallskip

\emph{Case 2.c:} Suppose that $x = \vA_{\mC-1} = \vA_\mC$. If $\CL(H_1)=2$, then this cannot happen because every $(\vA_\mC,v)$ is cooled by the end of round $\mC$. Hence, $\CL(H_1) \geq 3$. If $\CL(H_2) = 2$, then every $(\vA_\mC,z)$ is cooled at the end of round $\mC$ by $B_2$. Therefore, we have that $\CL(G\circ H_1)-1 \leq \mC \leq \CL(G \circ H_2)$ when $\CL(H_1) \geq 3$ and $\CL(H_2) = 2$. Finally, suppose that $\CL(H_2) \geq 3$ and let $\hat{w}$ be a vertex of $H_2$ that cools in round 3 with respect to $B_2$. We consider the two remaining cases for $(x,y)$ that is uncooled at the end of round $\mC$ in $G\circ H_1$. We have assumed $x = \vA_{\mC-1} = \vA_\mC$, and thus 
\begin{align*}
    d_{G\circ H_1}((\vA_{\mC-1},\vB_{\mC-1}),(x,y)) = d_{H_1}(\vB_{\mC-1},y) &\geq 2\\
    d_{G\circ H_1}((\vA_{\mC},\vB_{\mC}),(x,y)) = d_{H_1}(\vB_\mC,y) &\geq 1.
\end{align*}
In $G\circ H_2$, consider the vertex $(\vA_{\mC-1},\hat{w})$. Note that
\begin{align*}
    d_{G\circ H_2}((\vA_{\mC-1},g(\vB_{\mC-1})),(x,\hat{w})) = d_{H_1}(w,\hat{w}) &\geq 2\\
    d_{G\circ H_1}((\vA_{\mC},g(\vB_{\mC})),(x,\hat{w})) = d_{H_1}(w',\hat{w}) &\geq 1,
\end{align*}
which means that $(x,\hat{w})$ is uncooled at the end of round $\mC$ with respect to $B_2$. Hence, when $\CL(H_1),\CL(H_2) \geq 3$, we have $\CL(G\circ H_1) \leq \CL(G\circ H_2)$. This completes the proof of \emph{Case 2.c}. 

In summary, if $\CL(H_1) = 2$, then $\CL(G\circ H_1)\leq \CL(G\circ H_2)$, and if $\CL(H_1)\geq 3$, then
\begin{enumerate}
\item if $\CL(H_2) = 2$, then $\CL(G \circ H_1)-1\leq \CL(G\circ H_2)$, and\\
\item if $\CL(H_2) \geq 3$, then $\CL(G\circ H_1) \leq \CL(G \circ H_2)$.
\end{enumerate}

\noindent Swapping the roles of $H_1$ and $H_2$, we derive that $\CL(G\circ H_1) = \CL(G\circ H_2)$ when $\CL(H_1) = \CL(H_2) = 2$ or when $\CL(H_1),\CL(H_2)\geq 3$, and that $\CL(G\circ H_1) \leq \CL(G\circ H_2) \leq \CL(G\circ H_1)+1$ when $\CL(H_1)=2$ and $\CL(H_2) \geq 3$. 
\end{proof}

From Theorem \ref{thm:LEX:H}, the cooling number of $G\circ H$ can be computed using lexicographic product $\CL(G\circ \overline{K}_2)$, if $\CL(H)=2$, or using $\CL(G\circ \overline{K}_3)$, if $\CL(H)\geq 3$. This can greatly speed up computations, particularly if $H$ is large. Moreover, if $H$ is a clique, then $\CL(G \circ H)$ can be one of only two values. 

\begin{theorem}
If $G$ is a graph and $H$ is a clique, and both have order at least 2, then $\CL(G) \leq \CL(G \circ H) \leq \CL(G)+1$. In particular, if there exists a cooling sequence for $G$ of length $\CL(G)$, then $\CL(G \circ H) = \CL(G)+1$. Otherwise, $\CL(G \circ H) = \CL(G)$. 
\end{theorem}

\begin{proof}
Let $(\vA_1,\vA_2,\dots,\vA_\mA)$ be a cooling sequence for $G$ and denote by $H_i$ the copy of $H$ corresponding to $\vA_i$ for $1\leq i \leq \mA$. Recall that $\CL(G\circ H) \geq \CL(G)$ by Theorem~\ref{thm:LEX:circ_general}. 
        
If $\CL(G)=\mA$, then for some $v\in H$, consider the cooling sequence $\big((\vA_1,v),(\vA_2,v),$ $\dots,(\vA_\mA,v)\big)$ for $G\circ H$. The only cooled vertex in $H_\mA$ after round $\mA$ is $(\vA_\mA,v).$ The remaining vertices in $H_\mA$ (which exist since $|V(H)|\geq 2$) cool in round $\mA+1$. Thus, $\CL(G \circ H) \geq \CL(G)+1$. 

Suppose that $\CL(G \circ H) > \CL(G)+1$ and let $\big((y_1,z_1),(y_2,z_2),\dots,(y_k,z_k)\big)$ be a cooling sequence for $G\circ H$ with $k \geq \CL(G)+1$. We claim that $y_1,y_2,\dots,y_k$ must all be distinct vertices. As soon as a source is cooled in some $H_j$ in round $i$, all vertices in $H_j$ are cooled in the propagation step of round $i+1$, as they are all pairwise adjacent. Thus, no vertex in $H_j$ can be manually cooled in round $i+1$. 

We construct a copy of $G$ in $G\circ H$ containing all of the vertices $y_1,y_2,\dots,y_k$. Since $k \geq \CL(G)+1$, this copy of $G$ is cooled before $y_k$ is manually cooled, which is a contradiction. Thus, $k \leq \CL(G)$ and so $\CL(G \circ H) = \CL(G)+1$ when $\CL(G)=\mA$. 

Suppose instead that $\CL(G) = \mA+1$ and that there does not exist a cooling sequence for $G$ of length $\mA+1$. As above, let $\big((y_1,z_1),(y_2,z_2),\dots,(y_k,z_k)\big)$ be a cooling sequence for $G\circ H$ with $k \geq \mA+1$. Again, $y_1,y_2,\dots,y_k$ are distinct vertices of $G$, so we construct a copy of $G\circ H$ containing all of $y_1, y_2, \dots, y_k$. This is a cooling sequence for $G$ of length at least $\mA+1$, a contradiction. Thus, $k \leq \mA$ and $\CL(G\circ H) \leq \mA+1 = \CL(G)$. 
\end{proof}

If $H$ is not a clique, then we identify graphs for which $\CL(G \circ H)$ lies strictly between the bounds given in Theorem~\ref{thm:LEX:circ_general}. 

\begin{theorem} \label{thm:LEX:Pn}
If $n\geq 3$ and $H$ is not a clique, then $$\CL(P_n \circ H) \geq \CL(P_n)+1 = \left\lceil \frac{n+1}{2}\right\rceil +1 .$$
\end{theorem}

\begin{proof}
Let $V(P_n) = \{\vB_1,\vB_2,\dots,\vB_n\}$, with $\vB_i\vB_{i+1}\in E(P_n)$ for $1\leq i<n$, and let $x,y$ be nonadjacent vertices in $H$. Let $H_i$ be the copy of $H$ corresponding to $\vB_i$ for $1\leq i \leq n$. 

First, assume that $n$ is even. Note that 
$
\big((\vB_1,x),(\vB_1,y),(\vB_4,x),(\vB_6,x),\dots, (\vB_n,x)\big)
$ is a cooling sequence for $P_n \circ H$ of length $\frac{n}{2}+1 = \lceil \frac{n+1}{2}\rceil$. In round $\frac{n}{2}+1$, all of $H_{n-1}$ cools by propagation and $(\vB_n,x)$ is manually cooled. The rest of $H_n$ cools by propagation in the following round. Thus, $\CL(P_n \circ H) \geq \frac{n}{2}+2 = \left\lceil \frac{n+1}{2}\right\rceil +1$. 

Similarly, if $n$ is odd, then $\big((\vB_1,x),(\vB_1,y),(\vB_4,x),(\vB_6,x),\dots, (\vB_{n-1},x)\big)$ is a cooling sequence for $P_n \circ H$ of length $\frac{n-1}{2}+1 = \frac{n+1}{2}$. In round $\frac{n+1}{2}$, all of $H_{n-2}$ cools by propagation and $(\vB_{n-1},x)$ is manually cooled. In the following round, the rest of $H_{n-1}$ cools by propagation due to $H_{n-2}$, and $H_n$ is cooled by propagation due to $(\vB_{n-1},x)$. Thus, $\CL(P_n \circ H) \geq \frac{n+1}{2}+1 = \left\lceil \frac{n+1}{2}\right\rceil +1$.
\end{proof}

\begin{theorem} \label{thm:LEX:Cn}
If $n\equiv 0$ or $1 \pmod3$ with $n\geq 6$, and $H$ is not a clique, then $$\CL(C_n \circ H) \geq \CL(C_n)+1 = \left\lceil \frac{n+2}{3} \right\rceil +1.$$
\end{theorem}

\begin{proof}
Let $V(C_n)=\{\vB_1,\vB_2,\dots,\vB_n\}$, with $\vB_i\vB_{i+1} \in E(C_n)$ for $1\leq i<n$ and $\vB_1\vB_n\in E(C_n)$. Let $x,y$ be nonadjacent vertices in $H$, and let $H_i$ be the copy of $H$ corresponding to $\vB_i$ for $1\leq i\leq n$. 

First, let $n\equiv 0 \pmod{3}$ and $k=\frac{n}{3}$. Observe that
$
\big((\vB_1,x),(\vB_1,y),(\vB_4,x),(\vB_6,x),$ $\dots,(\vB_{2k},x) \big)
$
is a cooling sequence for $C_n\circ H$ of length $k+1$. In round $k+1$, all of $H_{2k-1}$ and $H_{2k+1}$ cools by propagation, in addition to $H_{2k-2}$. We then manually cool $(\vB_{2k},x)$ in this round and the rest of $H_{2k}$ cools by propagation in round $2k+2 = \left\lceil \frac{n+2}{3}\right\rceil +1$. 

Similarly, if $n\equiv 1 \pmod{3}$, let $k=\frac{n-1}{3}$. We then have that 
$
\big((\vB_1,x),(\vB_1,y),(\vB_4,x),$ $(\vB_6,x),$ $\dots,(\vB_{2k},x) \big)
$
is a cooling sequence for $C_n \circ H$ of length $k+1 = \frac{n+2}{3}$. In round $k+1$, all of $H_{\frac{2n}{3}-1}$ and $H_{\frac{2n}{3}+2}$ cool by propagation, in addition to $H_{\frac{2n}{3}-2}$. We then manually cool $(\vB_{2k},x)$ in this round, after which $H_{2k+1}$ and $H_{2k}$ cool by propagation in round $k+1 = \frac{n+2}{3}+1$. 

Thus, in both of these cases, $\CL(C_n \circ H) \geq \left\lceil \frac{n+2}{3}\right\rceil +1$. 
\end{proof}

If $n\equiv 2 \pmod{3}$, then the above strategy cools $C_n \circ H$ in $\CL(C_n) = \left\lceil \frac{n+2}{3} \right\rceil$ rounds. For both Theorem \ref{thm:LEX:Pn} and Theorem \ref{thm:LEX:Cn}, the lower bound is tight for small values of $n$ and $H = \overline{K}_2$. However, alternate methods would be needed to show equality in general. Exact values for lexicographic products of small paths and cycles are given in Table \ref{tab:LEX:lex_products_table}.

\begin{table}[ht]
\centering
\begin{minipage}{0.42\textwidth}
\centering
\begin{tabular}{|c|c|c|c|}
\hline
$\circ$ & $P_2$ & $P_3$ & $P_4$ \\
\hline
$P_2$ & 2 & 2 & 3 \\
\hline
$P_3$ & 3 & 3 & 3 \\
\hline
$P_4$ & 3 & 4 & 4 \\
\hline
$P_5$ & 4 & 4 & 5 \\
\hline
$P_6$ & 4 & 5 & 5 \\
\hline
$P_7$ & 5 & 5 & 6 \\
\hline
$P_8$ & 5 & 6 & 6 \\
\hline
\end{tabular}
\end{minipage}
\hspace{0.03\textwidth}
\begin{minipage}{0.42\textwidth}
\centering
\begin{tabular}{|c|c|c|c|}
\hline
$\circ$ & $C_3$ & $C_4$ & $C_5$ \\
\hline
$C_3$ & 2 & 2 & 3 \\
\hline
$C_4$ & 3 & 3 & 3 \\
\hline
$C_5$ & 3 & 3 & 3 \\
\hline
$C_6$ & 3 & 4 & 4 \\
\hline
$C_7$ & 4 & 4 & 4 \\
\hline
$C_8$ & 4 & 4 & 5 \\
\hline
\end{tabular}
\end{minipage}
\caption{The entry in row $i$, column $j$ contains the cooling number of the lexicographic product of the graph indicated in the first column of row $i$ and the first row of column $j$. }
\label{tab:LEX:lex_products_table}
\end{table}

Note that $\CL(P_n\circ H)$ and $\CL(C_n\circ H)$ may be very close to $\CL(P_n)$ and $\CL(C_n)$, respectively. However, there are graphs $G$ where $\CL(G\circ H)$ is much larger than $\CL(G)$. 

\begin{theorem}
\label{thm:LEX:centipede_lex}
For any $n\geq 2$, let $G_{2n+1}$ be the centipede with spine vertices given by $P_{2n+1}=(\vB_1,$ $\vB_2,$ $\dots,\vB_{2n+1})$ and a leaf $w_i$ on each $\vB_i$ for $i=4,6,\dots,2n$. We then have that $\CL(G_{2n+1})=\left\lceil \frac{3n+1}{2}\right\rceil$. Moreover, if $H$ is not a clique, then $\CL(G_{2n+1} \circ H) = 2n+1$.
\end{theorem}

\begin{proof}
First, assume $n$ is even and consider the following cooling sequence for $G_{2n+1}$:
$$
S=\big( \vB_1,\vB_3,w_4,w_6,\vB_7,w_8,w_{10},\dots,\vB_{2n-1},w_{2n}\big).
$$
When $\vB_{4i-1}$ is manually cooled in round $3i-1$, then $\vB_{4i}$ cools by propagation in the next round, and we manually cool $w_{4i}$ in round $3i$. This continues until $\vB_{2n-1}$ is manually cooled in round $\frac{3n}{2}-1$, after which we manually cool $w_{2n}$ in round $\frac{3n}{2}$. Finally, $\vB_{2n+1}$ cools by propagation in round $\frac{3n}{2}+1$. We note that $S$ is optimal because exactly one vertex cools in each propagation stage, the smallest possible. Thus, when $n$ is even, $$\CL(G_{2n+1}) = \frac{3n}{2}+1 = \left\lceil \frac{3n+1}{2}\right\rceil .$$ 

Now, assume $n$ is odd. We construct the following cooling sequence for $G_{2n+1}$: 
$$
\big( \vB_1,\vB_3,w_4,w_6,\vB_7,w_8,w_{10},\dots,\vB_{2n-3},w_{2n-2}, w_{2n},\vB_{2n+1}\big).
$$

Cooling proceeds as before, with $\vB_{2n+1}$ being manually cooled in round $\frac{3(n+1)}{2}-1$. Exactly one vertex is cooled in each propagation step, so when $n$ is odd, $$\CL(G_{2n+1}) = \frac{3(n+1)}{2}-1=\left\lceil\frac{3n+1}{2}\right\rceil .$$ 
        
Finally, let $x$ and $y$ be two nonadjacent vertices in $H$, and consider the following cooling sequence for $G_{2n+1}\circ H$ (which does not depend on the parity of $n$):
$$
\big((\vB_1,x),(\vB_1,y),(w_4,x),(w_4,y),(w_6,x),(w_6,y),\dots,(w_{2n},x),(w_{2n},y) \big).
$$
This sequence has length $2n$, and all $(\vB_{2n+1},x)$ are uncooled until round $2n+1$. Thus, $\CL(G_{2n+1} \circ H) \geq 2n+1$. Equality follows from Theorem~\ref{thm:GENE:diam_bounds}. 
\end{proof}

If $H$ is not a clique, then Theorem~\ref{thm:LEX:centipede_lex} provides a graph $G$ for which $\CL(G\circ H)$ is close to $\frac{4}{3}\CL(G).$

\section{Cooling direct products and disconnected graphs}

Whereas previous sections focused on bounding the cooling numbers of products of graphs in terms of the cooling numbers and diameters of their components, our first result in this section shows that such bounds are not always possible for the direct product. We do this by constructing a family of connected graphs $G\times H$ with unbounded diameter, where the diameters of $G$ and $H$ are fixed. Recall that $G\times H$ either has one or two components, and vertices $(x,y)$ and $(z,w)$ are connected in $G\times H$ if and only if there exists an $x$-$z$ walk in $G$ and a $y$-$w$ walk in $H$ of the same parity.

\begin{theorem}
\label{thm:DIR:direct_diam}
There is no function $f: \mathbb{N}\times\mathbb{N}\rightarrow\mathbb{R}$ such that, for any two graphs $G$ and $H$, the diameter of a component of $G\times H$ is bounded above by $f(\diam(G),\diam(H)).$
\end{theorem}

\begin{proof}
Let $G=P_n$ and $H=C_k$, where $n$ is even, $k$ is odd, and $k$ is sufficiently larger than $n$. Let $u$ and $v$ be the leaves of $G$, and let $x,y \in V(H)$ with $d_H(x,y) \geq 2.$ Any path from $(u,x)$ to $(v,y)$ follows an odd-length $u$-$v$ walk in $G$ in the first coordinate and an odd-length $x$-$y$ walk in $H$ in the second coordinate (both coordinates must change in each step since neither graph contains a loop). There is exactly one odd-length path connecting $x$ and $y$ in $H$, of length $k-2$, so $d_{G \times H}((u,x),(v,y)) =k-2.$

Now, construct $H^\prime$ from $H$ as follows. Order $V(H)$ so that $x$ is the ``leftmost'' vertex, followed by its shared neighbor with $y$, followed by $y$, and so on. Add an edge from $y$ to each vertex to the ``right'' of $y$ that is an even distance away (with respect to the ordering). These additions do not introduce any new odd-length $x$-$y$ paths (indeed, any $x$-$y$ path that uses a newly added edge has even length). Hence, $d_{G \times H}((u,x),(v,y)) =k-2$ and so $\diam(G\times H^\prime)\geq k-2$. 

Fix any function $f: \mathbb{N}\times\mathbb{N}\rightarrow\mathbb{R}$, and note that $\diam(H^\prime)=4$, so we have $f(\diam(G),\diam(H^\prime))=f(n-1,4)$. We choose $k\in\mathbb{N}$ large enough so that we have $k-2>f(n-1,4)$, implying $\diam(G\times H^\prime)>f(\diam(G),\diam(H^\prime))$, which yields the result.
\end{proof}

We now show that $\CL(G\times H)$ cannot be bounded above by any function of $\CL(G)$ and $\CL(H)$ (or of $\diam(G)$ and $\diam(H)$) for arbitrary graphs $G$ and $H$.

\begin{corollary}
\label{direct_infinite_differences}
For all $f,g: \mathbb{N}\times\mathbb{N}\rightarrow\mathbb{R}$ and for all graphs $G$ and $H$, the differences $\CL(G\times H)-f(\CL(G),\CL(H))$ and $\CL(G\times H)-g(\diam(G),\diam(H))$
can be arbitrarily large. 
\end{corollary}

\begin{proof}
Consider the graph $G\times H^\prime$ from the proof of Theorem~\ref{thm:DIR:direct_diam}. As $k$ increases, so does $\diam(G\times H^\prime)$. By Theorem~\ref{thm:GENE:diam_bounds}, we have $\CL(G\times H)\geq\lceil (\diam(G\times H^\prime)+2)/2\rceil$, which grows with $k$. However, $\diam(G)$ and $\diam(H)$ are constant with respect to $k$, so $g(\diam(G),\diam(H))$ is constant. Therefore, $\CL(G\times H)-g(\diam(G),\diam(H))$ can be made arbitrarily large.

Similarly, by Theorem~\ref{thm:GENE:diam_bounds}, $\CL(G)$ and $\CL(H^\prime)$ are bounded above and below by functions of $\diam(G)$ and $\diam(H^\prime)$, respectively. Since these diameters are constant with respect to $k$, if $a\leq \CL(G)\leq b$ and $c\leq \CL(H^\prime)\leq d$, then it follows that $f(\CL(G),\CL(H^\prime)) \in \{f(u,v):a\leq u \leq b,c \leq v \leq d\}.$ That is, $f(\CL(G),\CL(H^\prime))$ is bounded above by a constant. Again, since the lower bound on $\CL(G\times H)$ is increasing with $k$, the difference $\CL(G\times H)-f(\CL(G),\CL(H))$ is unbounded.
\end{proof}

Although burning is not the focus of this paper, the following corollary is immediate.

\begin{corollary}
\label{burn_direct_infinite_differences}
For all $f,g: \mathbb{N}\times\mathbb{N}\rightarrow\mathbb{R}$ and for all graphs $G$ and $H$, the differences $b(G\times H)-f(b(G),b(H))$ and $b(G\times H)-g(\diam(G),\diam(H))$
can be arbitrarily large. 
\end{corollary}

\begin{proof}
Observe that $b(G)$ is bounded above and below by functions of $\diam(G)$, and hence the same argument that was used to prove Corollary \ref{direct_infinite_differences} applies. In particular, $\lceil\sqrt{\diam{(G)}+1}\ \rceil\leq b(G)$ (see \cite{BJR}), and $b(G)\leq \CL(G)\leq \diam(G)+1$ (by Theorem~\ref{thm:GENE:diam_bounds}).
\end{proof}

As $G\times H$ is often disconnected, we turn our attention to expressing the cooling number of a disconnected graph in terms of the cooling numbers of its components.

\begin{theorem}
\label{thm:DIR:disconnected_equality}
Let $G$ be a disconnected graph with components $G_1,G_2,\ldots,G_n$, where $k\in\{0,1,\ldots,n\}$ is the number of components of $G$ that do not have an optimal cooling sequence where a source is cooled in the final round. We then have  $$\CL(G)=\begin{cases}
\sum_{i=1}^n\CL(G_i) & \text{if~} k=0, \\
\sum_{i=1}^n\CL(G_i)-k+1 & \text{otherwise.} 
\end{cases}$$ 
\end{theorem}

\begin{proof}
For the lower bound, we cool $G$ by optimally cooling each $G_i$ in sequence. If $k=0$, then the problem is straightforward. We begin by cooling $G_{1},G_{2},\ldots,G_{n-k}$ optimally one after the other, taking $\ell=\sum_{r=1}^{n-k}\CL(G_{r})$ rounds. We then cool $G_{n-k+1}$, $G_{n-k+2},$ $\ldots,G_{n}$ optimally one after the other, noting that no source in $G_{n-k+i}$ is available following the final propagation step in each $G_{n-k+i}$. We must therefore cool a source in $G_{n-k+i+1}$ in the same round that propagation occurs in $G_{n-k+i}$ and ``lose a round.'' 

The first source in $G_{n-k+1}$ is cooled in round $\ell+1$, the first source in $G_{n-k+2}$ is cooled in round $\ell+\CL(G_{n-k+1})$, the first source in $G_{n-k+3}$ is cooled in round $\ell+\CL(G_{n-k+1})+\CL(G_{n-k+2})-1$, and so on. Indeed, the first source in $G_{n}$ is cooled in round $\ell+\sum_{s=1}^{k-1}\CL(G_{n-k+s})-(k-2)$ and cooling lasts an additional $\CL(G_{n})-1$ rounds, for a total of $\ell+\sum_{s=1}^{k}\CL(G_{j_s})-k+1$. The result follows from the definition of $\ell$.

For the upper bound, let $\ell_i$ be the number of sources cooled in $G_i$ in an optimal cooling sequence $S$. Suppose $k=0$ and that the upper bound is false, so $\CL(G)>\sum_{i=1}^n\CL(G_i)$. Note that $\ell_i\leq \CL(G_i)$ for each $i$. We consider two cases.

\smallskip
    
\noindent\emph{Case 1:} A source is cooled in $G$ in the last round, so $\CL(G)=\sum_{i=1}^n \ell_i$. We have $$\sum_{i=1}^n\CL(G_i)<\CL(G)=\sum_{i=1}^n \ell_i\leq\sum_{i=1}^n\CL(G_i) ,$$ which is a contradiction.

\smallskip

\noindent\emph{Case 2:} A source is not cooled in $G$ in the last round, so $\CL(G)=\sum_{i=1}^n \ell_i + 1$. Since 
$$\sum_{i=1}^n\CL(G_i)<\sum_{i=1}^n \ell_i+1\leq\sum_{i=1}^n\CL(G_i)+1,$$ we must have 
$\sum_{i=1}^n \ell_i=\sum_{i=1}^n\CL(G_i)$.
Seeing as $\ell_i\leq\CL(G_i)$ for each $i$, we conclude $\ell_i=\CL(G_i)$ for $i\in\{1,2,\ldots,n\}$. We finish the $k=0$ case by considering two sub-cases.

\smallskip

\emph{Case 2.a:} $S$ never ``swaps between'' the components, so we cool $\ell_1$ sources in $G_1$, followed by $\ell_2$ sources in $G_2$, and so on. But since $\ell_i=\CL(G_i)$ for each $i$, this means $G$ cools in $\sum_{i=1}^n \ell_i<\sum_{i=1}^n \ell_i+1=\CL(G)$ rounds, contradicting the optimality of $S$.

\smallskip

\emph{Case 2.b:} $S$ ``swaps between'' the components. At some point during the cooling process, we cool a source in a $G_j$, followed by at least one source in other components of $G$, and then cool a source in $G_j$. Consider the cooling that occurs in $G_j$. There must be at least one propagation-only round where no source is chosen in $G_j$, but we cool $\ell_j=\CL(G_j)$ sources in $G_j$. Indeed, given any cooling sequence of length $\CL(G_j)$ in $G_j$, a source is cooled in the final round, and thus a propagation-only round is impossible. This is a contradiction, and the proof of the upper bound is complete when $k=0$.    
        
Finally, suppose that $k>0$ and the upper bound is false, so we have that $\CL(G)>\sum_{i=1}^n\CL(G_i)-k+1 $. Note that for $G_1,G_2\ldots,G_{n-k}$ we have $\ell_i\leq \CL(G_i)$, and for $G_{n-k+1},$ $G_{n-k+2},$ $\ldots,G_n$ we have $\ell_i\leq \CL(G_i)-1$. Hence, 
$
\sum_{i=1}^n\ell_i\leq\sum_{i=1}^n\CL(G_i)-k
$. 
If a source is cooled in $G$ in the last round, then $\CL(G)=\sum_{i=1}^n\ell_i$. Therefore,
$$
\sum_{i=1}^n\CL(G_i)-k+1<\CL(G)=\sum_{i=1}^n\ell_i\leq\sum_{i=1}^n\CL(G_i)-k ,
$$ 
a contradiction. If no source is cooled in the last round, then $\CL(G)=\sum_{i=1}^n\ell_i+1$ and
$$
\sum_{i=1}^n\CL(G_i)-k+1<\CL(G)=\sum_{i=1}^n\ell_i+1\leq\sum_{i=1}^n\CL(G_i)-k+1 ,
$$ 
which is also a contradiction, so the proof is complete.
\end{proof}

\begin{corollary}
\label{cor:DIR:disconnected_max}
Let $G$ be a disconnected graph with components $G_1,G_2,\ldots,G_n$. If each component $G_i$ is a maximally cool graph, then $\CL(G)=\sum_{i=1}^n\CL(G_i)-n+1$.
\end{corollary}

\begin{proof}
If $H$ is maximally cool, then no optimal cooling sequence has a source in round $\diam(H)+1$, since the rest of $G$. Theorem~\ref{thm:DIR:disconnected_equality} yields the result.
\end{proof}

A \emph{path-forest} is a disjoint union of paths. Since an optimal cooling sequence for $P_n$ has a source in the final round if and only if $n$ is odd, and $\CL(P_n)=\left\lceil\frac{n+1}{2}\right\rceil$ \cite{cooling}, we have the following corollary. 

\begin{corollary}\label{corH}
Let $G$ be the path-forest consisting of paths $P_{i_1},P_{i_2},\ldots,P_{i_n}$, where $k\in\{0,1,\ldots,n\}$ of the paths have an even number of vertices. We then have 
\[
\begin{aligned}
\CL(G)=
\begin{cases}
\sum_{j=1}^n \left\lceil \dfrac{i_j+1}{2} \right\rceil & \text{if } k=0,\vspace{0.5em} \\ 
\sum_{j=1}^n \left\lceil \dfrac{i_j+1}{2} \right\rceil - k + 1 & \text{otherwise}.
\end{cases}
\end{aligned}
\]
\end{corollary}

Notably, path-forests are a family of graphs where computing the burning number is NP-complete \cite{Burning_Hard}, but in contrast, computing the cooling number can be done in polynomial time by Corollary~\ref{corH}. We finish the section with three more applications of Theorem~\ref{thm:DIR:disconnected_equality}.

\begin{corollary}
If $G$ is a connected bipartite graph, then $$2\CL(G) - 1 \leq \CL(G \times K_2) \leq 2\CL(G).$$ 
\end{corollary}

\begin{proof}
Since $G$ is bipartite, the graph $G\times K_2$ is isomorphic to two disjoint copies of $G$. The result then follows from Theorem~\ref{thm:DIR:disconnected_equality}.   
\end{proof}

\begin{theorem}
If $G$ is a connected bipartite graph, then $$2\CL(G \circ \overline{K}_n)-1 \leq \CL(G \times K_{n,n}) \leq 2\CL(G \circ \overline{K}_n).$$
\end{theorem}

\begin{proof}
We show that $G \times K_{n,n}$ is isomorphic to two disjoint copies of the graph $G \circ \overline{K}_n$. 
Let $K_{n,n} = (U.V)$ and let $G=(A,B)$. Consider the subgraph of $G \times K_{n,n}$ induced by $(A\times U)\cup (B \times V)$. For a fixed $a\in A$, the vertices $\{(a,\vA_1),(a,\vA_2),\dots,(a,\vA_n)\} = (a,U)$ form an independent set of order $n$, or a copy of $\overline{K}_n$. The same is true for a fixed $y\in B$. 

Consider $xy\in E(G)$. Noting that $x\in A$ and $y\in B$, in $G \times K_{n,n}$, each $(x,\vA_i)$ is adjacent to each $(y,\vB_j)$, where $\vA_i \in U$ and $\vB_j \in V$. Hence, every vertex in $(x,U)$ is adjacent to every vertex in $(y,V)$. Moreover, if $ab\not\in E(G)$, then $(a,\vA_i)$ and $(b,\vB_j)$ are nonadjacent for $\vA_i\in U$ and $\vB_j \in V$. Thus, $(A\times U)\cup (B\times V)$ induces a copy of $G \circ \overline{K}_n$. 

Similarly, $(B\times U) \cup (A\times V)$ induces a copy of $G\circ \overline{K}_n$. Note that there are no edges between $A \times U$ and $A \times V$, because $A$ is an independent set. There are likewise no edges between $B \times V$ and $B \times U$, between $A \times U$ and $B \times U$, or between $B \times V$ and $A\times V$. Thus, the two copies of $G \circ \overline{K}_n$ are disjoint, and Theorem~\ref{thm:DIR:disconnected_equality} yields the result.
\end{proof}

\begin{theorem}
For all $n>1$, $\CL(P_n\times P_n)= 2n-1$.
\end{theorem}

\begin{proof}
The cases $n=2,$ $3,$ $4,$ and $5$ can be checked separately; see Table~\ref{tab:DIR:direct_product_table}. We now consider two cases based on the parity of $n$, where $n\geq 6$.
    
If $n$ is even, then the two components of $P_n\times P_n$ are isomorphic. Call them $G$ and $H$. We have that $\diam(G) = \diam(H) = n-1$. We show that $\CL(G)=\diam(G)+1=n$, and the result follows by Corollary~\ref{cor:DIR:disconnected_max}. 

We start by cooling vertices of the form $(x,1)$ in $G$, in order of increasing $x$. Note that in $G$ all $(x,1)$ have $x$ odd. In particular, we cool the vertex $u_i = (2i-1,1)$ in rounds $1 \leq i \leq \frac{n}{2}$. Note that $d_G(u_i,u_j) = 2|i-j| > |i-j|+1,$ so the $u_i$ satisfy the cooling condition in $G$.

Next, note that all $(y,n)$ in $G$ have $y$ even and that $d_G((y,n),u_i)=n-1$. Hence, they satisfy the cooling condition until round $1+(n-1)=n$. We therefore cool $u_j = (n-2j,n)$ in rounds $\frac{n}{2}+1\leq j \leq n-1$, with $(2,n)$ cooling in round $n$. Thus, $\CL(G)=\CL(H)=n$ and Corollary~\ref{cor:DIR:disconnected_max} yields the even case.

If $n$ is odd, then the components of $P_n\times P_n$ are not isomorphic. Let $G$ be the component containing the four leaves $(1,1)$, $(1,n)$, $(n,1)$, and $(n,n)$, and let $H$ be the other component. Observe that $\diam(G) = \diam(H) = n-1$. We show that $\CL(G)=\CL(H)=n$ and the result follows from Corollary~\ref{cor:DIR:disconnected_max}. 

We first describe a cooling strategy on $G$ that lasts $n$ rounds. Consider the sequence $\big((1,1)$, $(3,1)$, $\ldots$, $(n,1)$, $(n,n)$, $(n-2,n)$, $\ldots$, $(5,n)\big)$. We cool $\frac{n+1}{2}$ sources of the form $(x,1)$, in order of increasing $x$, and then $\frac{n+1}{2}-2$ sources of the form $(y,n)$, in order of decreasing $y$. In the following round, every vertex of $G$ is cooled, so the cooling process on $G$ using the above cooling sequence lasts $\frac{n+1}{2}+\frac{n+1}{2}-2 +1=n$ rounds. That the sources satisfy the cooling condition in $G$ follows from the even case above.   

Finally, we describe a cooling strategy on $H$ which lasts $n$ rounds. Consider the sequence $\big((2,1),(4,1), \ldots , \left(\tfrac{n-1}{2},1\right), (n,4), (n,6), \ldots , \left(n,\tfrac{n-1}{2}\right), (n-3,n)\big).$ We cool $\frac{n-1}{2}$ sources of the form $(x,1)$, in order of increasing $x$, and then $\frac{n-1}{2}-1$ sources of the form $(n,y)$, in order of increasing $y$ but skipping $y=2$. Finally, in round $n-1$, we cool $(n-3,n)$, noting that since $n>5$, the source $(n-3,n)$ does not end the cooling process in round $n-1$, as there is at least one uncooled vertex of the form $(z,n)$. Therefore, the cooling process on $H$ lasts $\frac{n-1}{2}+\frac{n-1}{2}-1+1+1=n$ rounds. As before, that the sources satisfy the cooling condition follows from the even case above. Thus, $\CL(G)=\CL(H)=n$ and the proof of the odd case follows by Corollary~\ref{cor:DIR:disconnected_max}. 
    \end{proof}

We finish by including in Table~4 exact values for the cooling numbers of direct grids, derived by computation.

\begin{table}[htpb!]
       \centering
       \begin{tabular}{|c|c|c|c|c|c|c|c|c|c|}
       \hline
           $\times$ & $P_2$ & $P_3$ & $P_4$ & $P_5$ & $P_6$ & $P_7$ & $P_8$ & $P_9$ & $P_{10}$ \\
           \hline
           $P_2$ & 3 & 4 & 5 & 6 & 7 & 8 & 9 & 10 & 11\\
           \hline
           $P_3$ & ~ & 5 & 6 & 7 & 8 & 9 & 10 & 11 & 12\\
           \hline
           $P_4$ & ~ & ~ & 7 & 8 & 9 & 11 & 12 & 13 & 15\\
           \hline
           $P_5$ & ~ & ~ & ~ & 9 & 10 & 12 & 13 & 14 & 15\\
           \hline
           $P_6$ & ~ & ~ & ~ & ~ & 11 & 13 & 14 & 16 & $\geq 17$\\
           \hline
       \end{tabular}
       \caption{The cooling number of $P_i \times P_j$.}
       \label{tab:DIR:direct_product_table}
   \end{table}

\section{Conclusion and Open Problems}

We studied the behavior of the cooling number under four different types of graph products. 
For the Cartesian, strong, and lexicographic product, we obtained general upper and lower bounds for the cooling number in terms of the cooling number of the factors.  For the direct products, we showed that such an upper bound derived as a function of the cooling number of the factors is impossible. We calculated the cooling number of the products of well-known graph families, such as trees and complete graphs. In particular, we strengthened and generalized several previously existing results on the cooling number of both Cartesian and strong grids. Finally, we also showed that cooling a disconnected graph reduces to optimally cooling each of its components in sequence. 

We finish with open problems. By Corollary~\ref{cor:CART:cool_centipedes}, all trees with no vertices of degree 2 are maximally cool. However, this is not a full characterization of maximally cool trees, as given by Lemma~\ref{lem:CART:cooling_centipedes_with_lobster_claws} and Theorem~\ref{thm:CART:cool_regular_trees}. 

In Theorem \ref{thm:LEX:circ_general}, we give an upper bound on the cooling number of $G\circ H$ in terms of the cooling number of $G$. While this upper bound is shown to be close to tight for general graphs, the difference is much larger for all known examples in which $G$ is a tree. We suspect that this bound can be improved, motivated by Theorem~\ref{thm:LEX:centipede_lex}. In particular, we conjecture that if $T$ is a tree and $H$ is a graph, then $$\CL(T\circ H) \leq \left \lfloor \frac{4}{3}\left( \CL(T)+1\right)\right\rfloor.$$ 

Another question relates to cooling lexicographic products. Let $\mathcal{G}_n$ be the set of all graphs with cooling number equal to $n$. Does there exist a $G \in \mathcal{G}_n$ such that $\CL(G \circ H) = n+k$ for each $k \in \left\{0,1,\ldots,\frac{n}{2}+1\right\}$? 

Analogous to Theorem~\ref{thm:STRONG:strong_lowerbounds_cartesian}, $\CL(G\boxtimes H)$ gives a lower bound on the cooling number of the direct product of $G$ and $H$. It is unclear whether this lower bound is tight or can be improved in general or for particular classes of graphs. 

Finally, it is known that burning is NP-complete on centipedes, as well as on path-forests \cite{Burning_Hard,liu1}. In Sections~3 and 6, we compute the cooling number of these graphs directly. This raises the question of the complexity of cooling. One open problem in this direction is determining the complexity of cooling spiders.

\section{Acknowledgments}

The first author was supported by an NSERC Discovery Grant, and the third author was supported by an NSERC CGRS D award.

\end{document}